\documentclass[aop]{imsart}

\usepackage[OT1]{fontenc}
\usepackage{amsmath}
\usepackage{amsthm}
\usepackage{amsfonts}
\usepackage{mathtools}
\usepackage{amssymb}
\usepackage{dsfont}
\usepackage{mathrsfs}
\usepackage{color}
\RequirePackage[colorlinks,citecolor=blue,urlcolor=blue]{hyperref}
\usepackage[colorlinks]{hyperref}
\usepackage[capitalize,nameinlink]{cleveref}
\usepackage{catchfilebetweentags}
\usepackage{enumerate}
\usepackage[numbers]{natbib}

\newcommand{\nn}{\nonumber}

\def\ML{\MoveEqLeft}

\theoremstyle{plain}

\newtheorem{theorem}{Theorem}[section]%
\newtheorem{corollary}{Corollary}[section]%
\newtheorem{lemma}{Lemma}[section]%
\newtheorem{proposition}{Proposition}[section]%
\newtheorem{remark}{Remark}[section]%
\newcommand{\ignore}[1]{}{}

\allowdisplaybreaks[3]

\numberwithin{equation}{section}
\def\nin{\not\in}
\def\P{\mathds{P}}
\def\P{{ \rm P}}
\def\E{\mathds{E}}
\def\Ex{{\rm E}}
\def\E{\Ex}
\def\Var{{\rm Var}\,}
\def\Cov{{\rm Cov}\,}

\def\I{\mathds{1}}
\def\R{\mathbb{R}}

\def\ss{\mathbb{S}}

\def\N{\mathbb{N}}
\def\F{\mathcal{F}}
\def\X{\mathscr{X}}

\def\Tr{{\rm Tr}}
\def\g{\psi}
\newcommand{\indicator}[1]{\I_{ #1 }}
\newcommand{\IN}[1]{\I_{\{#1\}}}

\newcommand{\ep}[1]{\E\left( #1 \right)}
\newcommand{\epb}[1]{\E\Bigl( #1 \Bigr)}
\newcommand{\conep}[2]{\E\left( #1 \,\middle\vert\, #2  \right)}
\newcommand{\conepb}[2]{\E\Bigl( #1 \,\Bigm\vert\, #2  \Bigr)}
\newcommand{\conepbb}[2]{\E\biggl( #1 \,\biggm\vert\, #2  \bigg)}

\newcommand{\inprod}[1]{\langle #1 \rangle}

\def\N{\mathcal{N}}

\def\x{\mathbf{x}}
\def\graphe{{E}}
\def\graphv{{V}}
\def\nl{\nonumber \\ }
\newcommand{\eq}[1]{$(\ref{#1})$}
\newcommand{\expb}[1]{\exp \left( #1 \right) }
\newcommand{\beq}{\begin{eqnarray*}}
  \newcommand{\eeq}{\end{eqnarray*}}
 \newcommand{\beqn}{\begin{eqnarray}}
 \newcommand{\eeqn}{\end{eqnarray*}}

\begin{document}

%%%%%%%%   Title part   %%%%%%%%%%%%%%%%%%%%%%%%%%%%%%%%%%%%%%%%%%%%%%%%%%%%%%%%%%%%%%%%%%%%%%%%%%%%%%%%%%%%%%%%%%%%%

%\providecommand{\keywords}[1]{\textbf{\textit{Key words---}}\textit{ #1}}
%\title{ A new Berry-Esseen bound for exchangeable pairs}
%\author{Zhuosong Zhang}
%\date{\today}
%\maketitle
%\begin{abstract}
%  A new  Berry-Esseen bound for exchangeable pairs is established without the
%  boundness assumption of $W-W'$. Optimal
%  convergence rate for normal and non-normal approximation can be achieved using this result. Our main result is
%  applied in many fields of models, such
%  as
%  quadratic forms, simple random sampling, general Curie-Weiss model,  mean field Heisenberg model,
%  colored graph model and limit theorem for prime factors.
%\end{abstract}
%\keywords{Stein's method, exchangeable pairs, Berry-Esseen bound,
%quadratic forms, simple random sampling, general Curie-Weiss model, mean field Heisenberg model,
%monochromatic edges, Erd\"os-Kac theorem}
%\tableofcontents

\begin{frontmatter}
\title{Berry-Esseen Bounds of Normal and Non-normal Approximation for Unbounded Exchangeable Pairs}%%\thanksref{T1}}
\runtitle{Berry-Esseen Bounds for Exchangeable Pairs}
%%\thankstext{T1}{Footnote to the title with the ``thankstext'' command.}

\begin{aug}
\author{\fnms{Qi-Man} \snm{Shao}\thanksref{t1,m1}\ead[label=e1]{qmshao@cuhk.edu.hk}}
\and
\author{\fnms{Zhuo-Song} \snm{Zhang}\thanksref{m1}\ead[label=e2]{zhuosongzhang@foxmail.com}}
%\and
%\author{\fnms{Third} \snm{Author}\thanksref{t1,m2}
%\ead[label=e3]{third@somewhere.com}
%\ead[label=u1,url]{http://www.foo.com}}

%%\thankstext{t1}{Some comment}
\thankstext{t1}{This research is partially supported by Hong Kong RGC GRF 14302515 and 403513}

%%\thankstext{t3}{Second supporter of the project}
\runauthor{Q. M. Shao and Z. S. Zhang}

\affiliation{The Chinese University of Hong Kong\thanksmark{m1}}
% and Another University\thanksmark{m2}}

\address{Department of Statistics\\
The Chinese University of Hong Kong\\
Shatin, New Territories\\
Hong Kong,  SAR of China
\\
\printead{e1}\\
\phantom{E-mail:\ }\printead*{e2}}

%\address{Address of the Third author\\
%Usually a few lines long\\
%Usually a few lines long\\
%\printead{e3}\\
%\printead{u1}}
\end{aug}

\begin{abstract}
  An exchangeable pair approach is  commonly taken in the normal and non-normal approximation using Stein's method. It has been successfully used to identify the limiting distribution and provide an error of approximation. However, when the difference of the exchangeable pair is not bounded by a small deterministic constant, the error bound is often not optimal. In this paper, using the exchangeable pair approach of Stein's method,  a new Berry-Esseen bound for an arbitrary random variable is established  without a bound on the difference of the exchangeable pair.  An optimal convergence rate for normal and   non-normal approximation is achieved when the result is applied to
  various examples including the  quadratic forms,  general Curie-Weiss model, mean field Heisenberg model and  colored graph model. % and number of prime factors.
\end{abstract}



\end{frontmatter}

%%%%%%%%%%%%%%%%%%   Introduction    %%%%%%%%%%%%%%%%%%%%%%%%%%%%%%%%%%%%%%%%%%%%%%%%%%%%%%%%%%%%%%%%%%%%%%%%%%%%%%%

\section{Introduction}

Let $W_n$ be a sequence of random variables under
study. Using the exchangeable pair approach of
Stein's method,  \citet{ChSh2011} and \citet{ShZh2016}, provided a concrete tool to identify the limiting distribution of $W_n$
as well as the $L_1$ bound (the Wasserstein distance) of the approximation. Our aim in this paper is to establish the Berry-Esseen type bound for the approximation.

Write $W = W_n$ and let $(W, W')$ be an exchangeable pair, that is, $(W,W')$ and $(W',W)$ have the same joint distribution. Put $\Delta = W - W'$.  For the normal approximation, assume that
\begin{equation*}
 \conep{\Delta}{W} = \lambda(W+R). \label{1.1-0}
\end{equation*}
Then, by  \citet{St86} (see also Proposition 2.4 in
\citet{ChSh11}), for any absolutely continuous function $h$ with $\|h'\| < \infty$,
\begin{equation*}
|\E h(W) - \E h(Z)| \leq 2 \|h'\| \Big( \E | 1
-\frac{ 1 }{ 2 \lambda}  \E(\Delta^2 | W)| + { 1 \over \lambda}  \E |\Delta|^3  + \E |R| \Big). \label{1.1-1}
\end{equation*}
Here and in the sequel, $Z$ denotes the standard normal random variable. For the Berry-Esseen bound, we have
\begin{equation}
  \sup_{z\in \R} \Big| P(W\leq z) - \Phi(z) \Big|\leq \E\Big| 1-{1\over 2\lambda} \E ( \Delta^2 | W)  \Big| + \E|R| + \Big({\E|\Delta|^3 \over \lambda}\Big)^{1/2},
   \label{St86}
\end{equation}
where $\Phi$ is the standard normal distribution function.
If in addition $|\Delta| \leq \delta$ for some
constant $\delta$, then by
\citet{rinott1997coupling} (see
also \citet{shao2006berry}),
\begin{equation}   \label{RR97}
\sup_{z\in \R} \Big| P(W\leq z) - \Phi(z) \Big|  \leq
  \E\Big| 1-{1\over 2\lambda} \conep{\Delta^2}{W} \Big| + \E|R| + 1.5 \delta + \delta^3/\lambda.
\end{equation}
It is known that \eqref{St86}  usually fails to provide an optimal bound. Similarly,  the bound in  \eqref{RR97}
may not be optimal unless $\delta$ is small enough. Hence, it would be interesting to seek an optimal Berry-Esseen bound
for an unbounded  $\Delta$. To this end,
\citet{ChenShao12} established the following  Berry-Esseen bound:
\begin{eqnarray}
  \label{chsh12}
  \lefteqn{ \sup_{z\in \R} |P(W\leq z)- \Phi(z)| }\nonumber\\
  & \leq & \E|R| + {1\over 4\lambda} \E(|W| + 1
  )|\Delta^3|  \nonumber\\ && +(1+\tau^2)
  \Big(4(1+\tau)\lambda^{1/2} + 6 \E\Big| {1\over
  2\lambda}\conep{\Delta^2}{W}-1 \Big| + {2\over
  \E[\Lambda]}\E|\Lambda-\E[\Lambda]|  \Big),
\end{eqnarray}
where $\Lambda$ is any random variable such that $\Lambda\geq \conep{\Delta^4}{W}$ and $\tau=\sqrt{\ep{\Lambda } }/\lambda $. They  obtained an optimal Berry-Esseen bound when the result was   applied to an independence tests by sums of squared sample correlation functions. However, \eqref{chsh12} is still  too complicated in general.

For the non-normal approximation, \citet{ChSh2011} developed similar results for both the $L_1$ bound and Berry-Esseen bound.

The exchangeable pair approach of Stein's method has been widely used in the literature. For example,
\citet{chatterjee2007multivariate},
\citet*{reinert2009multivariate} and \citet{Meckes2009Stein} established the
$L_1$ bounds for multivariate normal
approximation, and  \citet{Chat2005} and \citet{chatterjee2010}
obtained the concentration inequalities. We refer
to \citet{ChSh11} and \citet{chatterjee2014} for recent developments on Stein's method.

In this paper, we establish a new Berry-Esseen type bound for normal and non-normal approximation via exchangeable pairs. The bound is as simple as
$$
\E\Big| 1-{1\over 2\lambda} \conep{\Delta^2}{W} \Big| + \E | \E(\Delta |\Delta| \ | \ W)| + \E|R|,
 $$
 which yields an optimal bound in many applications.

The paper is organized as follows. The main
results are presented in Section \ref{main}. Section
 \ref{application} gives applications to the quadratic forms, general Curie-Weiss model, mean field Heisenberg model and  colored graph model. % and number of prime factors.
 The proof of the main results is given in
 Section \ref{proof}.  Other proofs of applications are postponed to Section  \ref{sec:proof}.

%%%%%%%%%%%%%%%%%%%%  Main results   %%%%%%%%%%%%%%%%%%%%%%%%%%%%%%%%%%%%%%%%%%%%%%%%%%%%%%%%%%%%%

\section{Main results}\label{main}
%\subsection{Main result}
In this section, we establish   Berry-Esseen bounds for normal and non-normal approximation via the exchangeable pair approach without the boundedness assumption.%%  requiring $|\Delta|\leq \delta$.

%The proof is put at the end of this noe.
%(Related results can be found in \citet*{St86,ChSh11,ChenShao12,chen2015,lachieze2015} and so on.)
%\subsection{Stein's method and exchangeable pairs}

\subsection{Normal approximation}

We first present a new  Berry-Esseen bound for normal
approximation,  which is a refinement of \eqref{St86}, \eqref{RR97} and \eqref{chsh12}.

\begin{theorem}\label{thm-1}
  Let $(W,W')$ be an exchangeable pair satisfying
  \begin{eqnarray}
    \conep{\Delta }{W} = \lambda (W+R), \label{con0}
  \end{eqnarray}
  for some constant $\lambda \in (0,1)$ and random variable $R$, where $\Delta = W-W'$.  Then,
  \begin{eqnarray*}
    \lefteqn{ \sup_{z\in \R} | \P(W\leq z) - \Phi(z)|  }\nonumber\\
    \label{res0} & \leq &  \E\Big|1- {1\over 2\lambda} \conep{\Delta^2}{W} \Big|  +  \E|R| +
    { 1 \over \lambda} \E | \E( \Delta \Delta^* |W) | ,
  \end{eqnarray*}
  where $\Delta^*:=\Delta^*(W,W')$ is any random variable satisfying $\Delta^*(W,W')=\Delta^*(W',W)$ and $\Delta^* \geq |\Delta|.$
\end{theorem}

The following two corollaries may be useful.

\begin{corollary}
\label{coro-1}
If $|\Delta| \leq \delta$ and $\E|W| \leq 2$, then
\begin{equation*}
     \sup_{z\in \R} | \P(W\leq z) - \Phi(z)|
     \leq   \E\Big|1- {1\over 2\lambda} \conep{\Delta^2}{W} \Big|  +  \E|R|  + 3 \delta.
  \label{coro-1a}
  \end{equation*}
\end{corollary}

Notice that the term $\delta^3/\lambda$  in \eqref{RR97} does not appear in the preceding corollary. One can check that under $|\Delta| \leq \delta$,
$$
\min\Big( 1, \E\Big|1- {1\over 2\lambda} \conep{\Delta^2}{W} \Big|  +  \delta\Big)
\leq 2 \min\Big(1, \E\Big|1- {1\over 2\lambda} \conep{\Delta^2}{W} \Big|  +  \delta^3/\lambda\Big).
$$
Hence, \cref{coro-1} is an improvement of \eqref{RR97} at the cost of assuming $\E|W| \leq 2$, which is easily satisfied.

It follows from the Cauchy inequality that for any $a>0$,
$$
|\Delta| \leq  a /2  + \Delta^2/(2a).
$$
Thus, we can choose $\Delta^* = a/2 + \Delta^2/(2a)$ with a proper constant $a$ and obtain the  following corollary.

\begin{corollary}
\label{coro-2}
Assume that $\E|W|\leq 2$. Then, under the condition of Theorem \ref{thm-1},
 \begin{eqnarray*}
    \lefteqn{ \sup_{z\in \R} | \P(W\leq z) - \Phi(z)|  }\nonumber\\
    \label{res01} & \leq &  \E\Big|1- {1\over 2\lambda} \conep{\Delta^2}{W} \Big|  +  \E|R| +
   2 \sqrt{ { \E| \E( \Delta^3 | W)| \over \lambda } }
. \label{coro-2a}
\end{eqnarray*}
 \end{corollary}

Clearly, $ \E| \E( \Delta^3 | W)| \leq
\E|\Delta|^3$. Hence,  \cref{coro-2} improves \eqref{St86}. In fact,
 \cref{coro-2} could yield an optimal bound while \eqref{St86} may not.

\subsection{Non-normal approximation}
In this subsection, we focus on the Berry-Esseen bound for non-normal
approximation.

  Let $W$ be a random variable satisfying $\P(a<W<b) =1$ where $-\infty \leq a<b\leq \infty.$ Let $(W,W')$ be an exchangeable pair satisfying %%$\ep{W}=0$,
  %%$\Var(W) =1$,
    \begin{eqnarray}
        \conep{W-W'}{W} = \lambda( g(W) + R),
    \label{con1}
  \end{eqnarray}
where $g$ is a measurable  function with domain $(a,b)$, $\lambda\in (0,1)$ and $R$ is a random variable.

\ignore{Let
$$
G(y) = \int_{w_0}^y g(t) dt,
$$
let $Y$ be a random variable with the probability density function (p.d.f.)
\begin{eqnarray}
  \label{pdf}
  p(y) = c_1 e^{-G(y)},
\end{eqnarray}
where $c_1$ is the  normalizing  constant.
}

Assume that $g$ satisfies the following conditions:
\begin{itemize}
\label{page-con}
  %%\item $\E|g(W)|<\infty$;
  %\item $g(b-) = \infty $ or $g(b-)\le 0$;
  \item[(A1)]  $g$ is non-decreasing, and there exists $w_0 \in (a, b)$ such that $(w-w_0)g(w) \geq 0$ for $w\in (a,b)$;
  \item[(A2)]  $g'$ is continuous and $2(g'(w))^2-g(w)g''(w)\geq 0$ for all $w\in (a,b)$; and
  \item[(A3)]  %$G(a+ ) = G(b-) = + \infty.$ %
      $\lim_{y \downarrow a} g(y) p(y) = \lim_{y \uparrow b} g(y) p(y) =0$,
      where
      \begin{equation}
          \label{pdf}
          p(y ) = c_1 e^{-G(y)}, \quad G(y) = \int_{w_0}^y g(t) dt,
      \end{equation}
      and $c_1$ is the constant so that $\int_a^b p(y) dy = 1$.
\end{itemize}
Let $Y$ be a random variable with the probability density function (p.d.f.) $p(y)$,
and let $\Delta = W -W'$.

  \begin{theorem} \label{thm-2}
 We have
  \begin{eqnarray}
  \lefteqn{\sup_{z\in \R} |  \P(W\leq z) - \P( Y \leq z)  |  }\nonumber\\
  & \leq &  \E\Big| 1-\frac1{2\lambda} \conep{\Delta^2}{W} \Big| + \frac{1}{\lambda}\E\left| \conep{\Delta\Delta^*}{W}\right| + \frac{1}{c_1} \E|R|,
  \label{res1}
\end{eqnarray}
where $\Delta^*:=\Delta^*(W,W')$ is any random variable satisfying $\Delta^*(W,W')=\Delta^*(W',W)$ and $\Delta^* \geq |\Delta|.$
  %\label{thm1}
\end{theorem}
To make the bound meaningful, one should choose $\lambda \sim (1/2) E(\Delta^2)$.
It is easy to see that $g(w) = w$ satisfies conditions (A1)--(A3). More generally, (A1)--(A3) are also satisfied for
$g(w) = w^{2k-1}$, where $k\geq 1$ is an integer.

\ignore{

Moreover, we present one corollary for the case when $g(w)=c_2 w^{2k-1}$ where $k\geq 1$.
\begin{corollary}
  \label{cor2}
  Let $(W,W')$ be an exchangeable pair satisfying
  \begin{eqnarray}
    \conep{W-W'}{W} = \lambda (2c_2k W^{2k-1}+R),
  \end{eqnarray}
  for some constant $\lambda \in (0,1),c_2> 0, k\geq 1$ and random variable $R$.
  Let $Y$ be a random variable with p.d.f.
  \begin{eqnarray}
    p(y) = c_1 e^{-c_2 y^{2k}}
  \end{eqnarray}
  where $c_1$ is the normalization constant.
  Let $\Delta = W-W'$, then
  \begin{eqnarray}
    \lefteqn{ \sup_{z\in \R} | \P(W\leq z) - F(z)|  }\nonumber\\
    & \leq & 2 \E\Big|1- {1\over 2\lambda} \conep{\Delta^2}{W} \Big| + {2\over \lambda} \E|\conep{\Delta\Delta^*}{W} | + {1\over c_1} \E|R|
  \end{eqnarray}
  where $F(z)$ is the distribution function of $Y$ and  $\Delta^*$ is as defined in \cref{thm1}.
\end{corollary}
\begin{proof}
  Letting $g(w) = 2c_2k w^{2k-1}$, it is sufficient to show that $g(w)$ satisfies conditions A1)-A5) in \cref{thm1}. In fact,
  we have $g(w)$ is increasing on the real line, $wg(w)\geq 0$ for all $w\in \R$, and
  \begin{eqnarray*}
    \lim_{w\rightarrow \pm\infty} g(w) = \pm \infty.
  \end{eqnarray*}
  Moreover,
  observe that
  \begin{eqnarray*}
    g'(w) = 2k(2k-1) c_2 w^{2k-2}, \quad g''(w) = 2k(2k-1)(2k-2)c_2 w^{2k-3},
  \end{eqnarray*}
  then $2(g'(w))^2 \geq g(w)g''(w)$ for all $k\geq 1$.
}

\section{Applications}\label{application}

In this section, we give some applications for our main result.

\subsection{Quadratic forms}
We first consider a classical example as a simple application.
Suppose $X_1,X_2,\cdots$ are i.i.d. random varaibles with a zero mean, unit
variance and a finite fourth moment. Let $A=\left\{ a_{ij} \right\}_{i,j=1}^n$ be
a real symmetric matrix and let
$W_n=\sum_{1\leq i\neq j\leq n}a_{ij}X_iX_j$.
The central limit theorem for $W_n$ has been extensively discussed in the literature.
%%There are many literatures considering this problem. The conditions for normal
%%approximation can be found in \cite{Rotar86,HALL1984},  and \cite{dejong87}.
For example, \citet{dejong87} %%\citet*{dejong87}
used $U$-statistics  and  proved a central limit theorem for $W_n$ when
$$
\sigma_n^{-4} \Tr(A^4) \to 0  \  \  \mbox{and} \ \ \sigma_n^{-2}\max_{1\leq i\leq n} \sum_{1\leq j\leq n} a_{ij}^2
\to 0,
$$
where $\sigma_n^2 = 2  \Tr(A^2) = \Var(W_n)$.
An $L_1$ bound was given by \citet{chatterjee2008}
while \citet{GoTi2002} gave a Kolmogorov distance with a
convergence rate $\lambda_1/\sigma_n $, where $\lambda_1$ the largest absolute eigenvalue of $A$.

%%\cite{gotze1999,GoTi2002,chatterjee2008}. For example, \citet*{chatterjee2008} used a new method of normal %%approximation to give a $L_1$ bound for $W_n$ with that $(X_1,X_2,...)$ is a Rademacher sequence, \citet*{gotze1999} %%gave a Kolmogorov distance for normalized $W_n$ with convergence rate $|\lambda_1|$ where $\lambda_1$ is the largest %%eigenvalue of $A$.
Here,  we apply Theorem \ref{thm-1} and obtain the following result.

%We now present a Berry-Esseen bound of $W_n$ using exchangeable pairs, which is as stated in the following theorem.
\begin{theorem}  \label{thma1}
  Let $X_1,X_2,\cdots$ be i.i.d. random variables with a zero mean, unit variance
  and a finite fourth moment. Let $A=(a_{ij})_{i,j=1}^n$ be a real symmetric
  matrix with $ a_{ii}=0$ for all $ 1\leq i\leq n$ and $\sigma_n^2 = 2 \sum_{i=1}^n\sum_{j=1}^n
  a_{ij}^2.$ Put $W_n ={1\over \sigma_n } \sum_{i\neq j} a_{ij}X_iX_j$. Then,
  \begin{eqnarray}
    \lefteqn{\sup_{x\in \R}\left|  \P(W_n\leq x)
    -\Phi(x)  \right|}\nn\\
    &\leq&  \frac{C \E X_1^4 }{\sigma_n^2 }  \Big( \sqrt{ \sum_{i}\Big(\sum_j
     a_{ij}^2 \Big)^2}+
 \sqrt{\sum_{i,j}\Big(\sum_{k} a_{ik}a_{jk}
 \Big)^2 }    \Big),
    \label{qraform}
  \end{eqnarray}
  where $C$ is an absolute constant.
\end{theorem}

%%\begin{remark}
  It is easy to check that
  \begin{align*}
      \sum_{i,j}\Big(\sum_k a_{ik}a_{jk}\Big)^2& =  \Tr(A^4),
   \intertext{and}
   \sum_{i} \Big(\sum_j a_{ij}^2 \Big)^2 & \leq
   \max_{1\leq i\leq n} \sum_{j}a_{ij}^2 \sigma_n^2  \leq
   \lambda_1^2\sigma_n^2,
  \end{align*}
  which means that the first term in (\ref{qraform}) is
  less than the bound $\lambda_1/\sigma_n$ given in
  Theorem 1 of \citet{GoTi2002}.
  However,  comparing it with the $L_1$
 bound
 given in \citet{chatterjee2008}, we conjecture that the bound in \eqref{qraform} can be improved to
 \begin{eqnarray*}
    \sup_{x\in \R}\left|  \P(W_n\leq x) -\Phi(x)  \right|
    \leq C  \left( \frac{1}{\sigma_n^4} \sum_{i}\Big(\sum_j a_{ij}^2 \Big)^2+ \frac{1}{\sigma_n^2} \sqrt{\sum_{i,j}\Big(\sum_{k} a_{ik}a_{jk} \Big)^2 }    \right).
      \end{eqnarray*}

\subsection{General Curie-Weiss model}
The Curie-Weiss model has been extensively discussed in  the statistical physics field. The asymptotic behavior for the Curie-Weiss model was studied by Ellis and Newman \cite{el78a,El78b,El78c}. Recently, Stein's method has been used to obtain the
convergence rate of the Curie-Weiss model. For
example,     \citet{ChSh2011}
 used exchangeable pairs to get a Berry-Esseen bound at the critical temperature  of
the simplest Curie-Weiss model, where the magnetization was valued on $\left\{
-1,1 \right\}$ with equal probability; and
\citet{CFS13} and  \citet{ShZhZh2015} %\citet*{ChSh2013,ShZhZh2015}
established the Cram\'er type  moderate
deviation result for non-critical and critical temperature, respectively.
More generally,  when the magnetization was distributed as a measure $\rho$ with a finite support,
\citet{chatterjee2010} obtained an exponential probability inequality. In this subsection, we apply
 Theorem \ref{thm-1} to establish a Berry-Esseen bound for the  general Curie-Weiss model.

Let $\rho$ be  a probability measure satisfying
\begin{eqnarray}
    \int_{-\infty}^{\infty} x d\rho(x) = 0, \text{ and } \int_{-\infty}^{\infty} x^2 d \rho(x) =1. \label{cwmc0}
\end{eqnarray}
$\rho$ is said to be type $k$ (an integer) with strength $\lambda_\rho$ if
\begin{eqnarray*}
    \int_{-\infty}^\infty x^j d \Phi(x) - \int_{-\infty}^\infty x^j d\rho(x) = \begin{cases}
    0, & \mbox{for } j = 0, 1, ..., 2k-1, \\
    \lambda_\rho >0, & \mbox{for } j = 2k,
  \end{cases}
\end{eqnarray*}
where $\Phi(x)$ is the standard normal distribution function.

We define the Curie-Weiss model as follows. For a given measure $\rho$, let $(X_1,\cdots, X_n)$ have the joint p.d.f.
\begin{eqnarray}
  \label{cwmpdf}
  dP_{n,\beta} (\x) = {1\over Z_n} \exp\Big(
  {\beta(x_1+\cdots + x_n)^2 \over 2n} \Big)
  \prod_{i=1}^n d\rho(x_i),
\end{eqnarray}
where $\x=(x_1,\cdots , x_n)$, $0<\beta\leq 1$ and $Z_n$ is the normalizing constant.

Let $\xi$ be a random variable with probability measure $\rho$.
Moreover, assume that
\begin{enumerate}[(i)]
    \item for $0< \beta < 1$, there exists a
constant $b > \beta$ such that
\begin{eqnarray}\label{c1-00}
    \E e^{t \xi } \leq e^{t^2 \over 2b}, \quad \text{for } -\infty< t < \infty;
\end{eqnarray}

\item for $\beta=1$,  there exist constants $b_0>0, b_1 >0$ and
$b_2 >1$ such that
\begin{eqnarray}
  \label{c1-01}
  \E e^{t\xi} \leq
  \begin{cases}
    \expb{ t^2/2 - b_1 t^{2k} }, & |t| \le  b_0,\\
    \expb{ \frac{t^2}{2b_2} } , & |t| > b_0.
  \end{cases}
\end{eqnarray}
\end{enumerate}
%\end{enumerate}
Let $S_n = X_1+\cdots + X_n$.   \citet{El78b,El78c}  showed that
\begin{itemize}
  \item[(i)]  if $0 < \beta <1$, then $n^{-1/2} S_n$ converges to a normal distribution
    $\N(0,(1-\beta)^{-1})$; and
  \item[(ii)]   if $\beta =1$, and $\rho$ is of
      type $k$, then $n^{-1+{1\over 2k}} S_n$ converges to a non-normal distribution with p.d.f.
      \begin{eqnarray*}
          p(y) = c_1 e^{-c_2 y^{2k}},
      \end{eqnarray*}
      where $c_2>0$ and $c_1$ is the normalizing constant.
\end{itemize}

The following theorem provides the rate of convergence.

\begin{theorem}
  \label{cwmthm1}
  Let $(X_1,\cdots, X_n)$ follow the joint p.d.f.
  \eqref{cwmpdf},  where $\rho$ satisfies
  \eqref{cwmc0}.

  \begin{itemize}
    \item[(i)]   If $0< \beta <1$ and
      \eqref{c1-00} is satisfied, then for $W_n = n^{-1/2} S_n$, we have
      \begin{eqnarray}\label{cwmt1-1}
    \sup_{z\in \R} |\P(W_n\leq z) - \P(Y_1\leq
    z)| \leq C n^{-1/2}.
  \end{eqnarray}
  where $Y_1\sim \N(0,{1\over 1-\beta})$ and $C$
  is a constant depending on $b$ and $\beta.$

\item[(ii)] If  $\beta =1$, $\rho$ is of type
    $k$   and \eqref{c1-01} is
  satisfied,
  then   for $W_n= n^{-1+ \frac{1}{2k}}S_n$, we have
  \begin{eqnarray}\label{cwmt1-2}
    \sup_{z\in \R} | \P(W_n\leq z) - \P(Y_k \leq
    z)| \leq C n^{-{1\over 2k}},
  \end{eqnarray}
  where $C$ is a constant depending on
  $b_0,b_1,b_2$ and $k$; the density function of $Y_k$ is given by
  \begin{eqnarray*}
      p(y) = c_1 e^{-c_2 y^{2k}}, c_2 = {H^{(2k)}(0) \over (2k)!};
  \end{eqnarray*}
  and $c_1$ is the normalizing constant and $H(s) = s^2/2 - \ln ( \int_{-\infty}^\infty \exp(sx) d \rho(x))$.
  \end{itemize}
\end{theorem}

\subsection{Mean field Heisenberg model}
%%In statistical physics, the Heisenberg model is one of the most important statistical mechanical models of ferromagnetism.
The Heisenberg model is a statistical model for the phenomena of ferromagnetism and antiferromagnetism in the study of magnetism theory.

Let $G_n$ be a finite complete graph with $n$ vertices. At each site of the graph is a spin in $\ss^2$, so the state space is $\Omega_n = (\ss^2)^n$ with $P_n$ the $n$-fold product of the uniform probability measure on $\ss^2$. The mean field Hamiltonian energy of the Heisenberg model  $H_n:\Omega_n \mapsto \R$ is
\begin{eqnarray*}
  H_n(\sigma) = -{1\over 2n} \sum_{1\leq i,j\leq n} \inprod{\sigma_i,\sigma_j},
\end{eqnarray*}
where $\inprod{\cdot,\cdot}$ is the inner product in $\R^3.$
The Gibbs measure $P_{n,\beta}$ is given by the density function
\begin{eqnarray*}
  dP_{n,\beta} = {1\over Z_{n,\beta}} \exp\Big( {\beta\over 2n} \sum_{1\leq i,j\leq n} \inprod{\sigma_i,\sigma_j} \Big)={1\over Z_{n,\beta}}\exp(-\beta H_n(\sigma)),
\end{eqnarray*}
where $Z_{n,\beta} = \int_{\Omega_n} \exp(-\beta H_n (\sigma)) dP_{n}$.

%%Our aim is to study the asymptotic behavior of the total spins $\sum_{i=1}^n \sigma_i$.
Consider the random variable
\begin{eqnarray}\label{def:W}
  W_n = \sqrt{n}\Big( {\beta^2\over n^2 \kappa^2} \Big|\sum_{j=1}^n \sigma_j \Big|^2 -1 \Big),
\end{eqnarray}
where $|\cdot|$ is the Euclidean norm in $\R^3$ and $\kappa$ is the solution to the equation
\begin{eqnarray}\label{fun:g}
  x/\beta= (\coth(x) -1/x). %%:=g(x).
\end{eqnarray}

Let $\g(x) = \coth(x) - 1/x$ and
 \begin{eqnarray}
  B^2 = {4\beta^2 \over (1-\beta \g'(\kappa))\kappa^2} \Big({1\over \kappa^2} - {1\over \sinh^2(\kappa)} \Big).
  \label{def:B}
\end{eqnarray}
\citet{KirMec2013} showed  that when $\beta>3$,
$W_n/B$ converges to a standard normal distribution with an $L_1$ bound $O(\log(n) n^{-1/4})$.
They also showed that
when $\beta =3$,
the random variable $T_n = c_3n^{-3/2} |\sum_j \sigma_j|^2$, where $c_3$ is a constant such that the variance of $T_n$ is $1$,
converges in distribution to  $Y$ with the density function
\begin{eqnarray*}
  p(y) = \begin{cases}
      C y^5 e^{-3y^2/(5c_3)}, & y\geq0, \\
      0, & y<0,
  \end{cases}
\end{eqnarray*}
where $C$ is the normalizing constant.
\ignore{ They also obtained the convergence rate that
\begin{eqnarray*}
  \sup_{\|h''\|_{\infty} \leq 1} \big|\ep{h(T_n)}- \ep{h(Y)}\big| \leq C  \log( n) n^{-1/2}
\end{eqnarray*}
where $C$ is a universal constant.
}

The following theorem gives a Berry-Esseen bound for the case $\beta >3$. The case $\beta=3$ will be studied in another paper.

\begin{theorem}\label{thmhm}
    Let  $W_n$ be the random variable defined as in \eqref{def:W} and $B$ as  in \eqref{def:B} with $\beta >3$. Then, we have   \begin{eqnarray}\label{res:hm}
      \sup_{z\in \R}|\P(W_n/B\leq z) - \Phi(z)|\leq c_{\beta} n^{-1/2},
  \end{eqnarray}
  where $c_{\beta}$ is a constant depending on $\beta.$
\end{theorem}

%%It was shown in \cite{KirMec2013} that the $L_1$ bound is $O((\log(n)/n)^{1/4})$, which  is actually not the optimal %%rate, and we successfully improve it to the optimal one in our work.

\subsection{Counting monochromatic edges in uniformly colored graphs}
%%In this section, we consider the counting monochromatic edges in uniformly
%%colored graphs.
The study of monochromatic and heterochromatic subgraphs of an edge-colored graph
 dates back to the 1960s, and the last two decades has witnessed a significant development in the study of normal and Poisson approximation.
 %%yet the development was not
%%so fast. Since the 1990��s, many important results appeared. For example, people considered the normal approximation and Poisson approximation for this model, where Stein's method showed its strong advantage.
%Barbour, Holst and Janson (1992) %%
 \citet*{barbour1992poisson}
 used Stein's
method to show that the number of monochromatic edges for the complete graph converges to
a Poisson distribution.  %Arratia, Goldstein and Gordon (1990) %%
\citet*{arratia1990poisson}
applied Stein's method %% by dependency graphs
to prove a Poisson approximation theorem for the number of monochromatic cliques in a uniform
coloring of the complete graph. We refer to
\citet*{chatterjee2005exchangeable} and \citet{cerquetti2006poisson}
%and Cerquetti and Fortini (2006)
for other related results.
  %%Also, \citet*{chatterjee2005exchangeable} used Stein's method and exchangeable pairs to consider the dependency graph approach. Stein's method is again used by \citet*{cerquetti2006poisson} to study the Poisson limit theorems for the
%%number of general monochromatic subgraphs in a random coloring of a graph sequence.%%

In this subsection, we consider normal approximation for the counting of monochromatic edges in uniformly colored graphs.
Let $G=\left\{ \graphv(G), \graphe(G) \right\}$ be a simple undirected graph, where
$\graphv(G)=\left\{ v_1,\cdots,v_n \right\}$ is
the vertex set and $\graphe(G)$ is the
edge set. For $1 \leq i \leq n$, let
$$A_i = \left\{1 \leq  j\leq n, j\neq i,  (v_i,v_j) \in
\graphe(G) \right\}
$$ be the
neighbourhood of  index $i$ and $d_i = \#(A_i)$ be the number of edges connected
to $v_i$. Denote   the total number of edges in $G$ by $m_n$, which is equal to $\sum_{i=1}^n d_i/2$.
Each vertex  is colored independently and uniformly with
$c_n \geq 2$ colors, denoted by $\xi_i$ the color of $v_i$. Let
$Y_n$ be the number of monochromatic edges in
$G_n$. \citet{rinott1996} proved
the  central limit theorem for $Y_n$ while
\citet{Fang2015} obtained
 %%$(Y_n - \ep{Y_n})/\sqrt{\Var(Y_n)}$ converges to a standard
%%normal distribution (see \cite{rinott1996,bhattacharya2013}). Moreover,
%%in \cite{Fang2015}, the author also obtained the
the Wasserstein distance with an order of $\sqrt{c_n/m_n}+c_n^{-1/2} $. The following theorem provides  a Berry-Esseen bound.
\ignore{
Let $\left\{ G_n, n\geq 1 \right\}$ be a sequence of simple graphs with
vertices $\left\{ v_1,\cdots,v_n \right\}$, and for every $1\leq i\leq n$,
there are $d_i$ edges connected to $v_i$. Thus in $\left\{ G_n \right\}$,
there
are in total $m_n:={1\over 2}\sum_{i=1}^n
d_i$ edges. Also, each edge is colored independently and uniformly with
$c_n$ colors. Let $Y_n$ be the number of monochromatic edges in $G_n$. Rinott and Rotar (1996) proved that
the  central limit theorem for $Y_n$ while
\citet{Fang2015} obtained
 %%$(Y_n - \ep{Y_n})/\sqrt{\Var(Y_n)}$ converges to a standard
%%normal distribution (see \cite{rinott1996,bhattacharya2013}). Moreover,
%%in \cite{Fang2015}, the author also obtained the
the Wasserstein distance with an order of $\sqrt{c_n/m_n}+c_n^{-1/2} $. The following theorem provides  a Berry-Esseen bound.
 %%for this approximation using Stein's method.
%%However, we still don't know the Berry-Esseen bound for it.

Let $G=\left\{ V(G), E(G) \right\}$ be a simple undirected graph, where
$V(G)=\left\{ v_1,\cdots,v_n \right\}$ is the vertex set and $E(G)$ is the
edge set. Let
$A_i = \left\{ j\leq n, j\neq i,  (v_i,v_j) \in E(G) \right\}, 1\leq i\leq
n$ be the
neighbour index of $i$ and $d_i = \#(A_i)$ be the number of edges connected
to
$v_i$. Define $m=\sum_{i=1}^n d_i/2$, the total number of edges in $G$. We
color each vertex of $G$ independently and uniforly at random with $c>2$
colors and denote by $\xi_i$ the color of $v_i$. Using those notations, we
have
the following theorem.
}

\begin{theorem}
  \label{thm:me}
  Let
  \begin{eqnarray*}
  W_n={1\over 2} \sum_{i=1}^n \sum_{j\in A_i} {\IN{\xi_i = \xi_j} -{1\over c_n}
  \over \sqrt{ {m_n\over c_n}(1-{1\over c_n}) }}. \label{meW}
\end{eqnarray*}
Then,
\begin{eqnarray*}
  \sup_{z\in \R} \left|  \P(W_n\in z) - \Phi(z)  \right| & \leq &  C
  (\sqrt{1/c_n} + \sqrt{d_n^*/m_n}+\sqrt{c_n/m_n}),
  \label{res:me}
\end{eqnarray*}
where $C$ is an absolute constant  and  $d_n^*=\max\left\{ d_i, 1\leq i\leq n
\right\}.$
 %% \label{thm:me}
\end{theorem}
\section{Proof of main results} \label{proof}

As the normal approximation is a special case of the non-normal approximation, we prove Theorem \ref{thm-2} only. The only difference for the normal approximation is that the Stein's solution can be bounded by $1$ instead of $\sqrt{2 \pi}$.

Let $Y$ be the random variable  with the p.d.f. $p(y)$ defined in \eq{pdf}. For a given $z$, let $f:=f_z$ be the solution to the  following Stein equation:
\begin{equation}
    f'(w)-g(w)f(w) = \indicator{ \left\{ w\leq z \right\}} - F(z), \quad z\in (a,b)\label{eq1},
\end{equation}
where $F$ is the distribution function of $Y$. It
is known (see, e.g., \citet{ChSh2011})
that
\begin{equation}
      f_z(w) = \left\{ \begin{array}{ll}
              {F(w)(1-F(z))\over p(w)}, \ \ & w\leq z,\\
         {F(z)(1-F(w)) \over p(w)}, & w>z.
      \end{array}
      \right.
      \label{fz}
      \end{equation}
We first prove some basic properties of $f_z$.

\begin{lemma} \label{l4.1}
    Suppose that conditions (A1)--(A3) are satisfied.  Then,
\begin{gather}
0 \leq f_z(w) \leq 1/c_1,
\label{l4.1-1} \\
\|f_z'\| \leq 1,
\label{l4.1-2}\\
    \|g f_z\| \leq 1 , \label{l4.1-3}
\intertext{and}
\label{l4.1-4}
    g(w) f_z(w) \text{ is non-decreasing.}
\end{gather}
\end{lemma}

We remark that when $g(w) =w$, i.e., for the  normal approximation, it is known that $ 0 \leq f_z(w) \leq 1$ (see, e.g., Lemma 2.3 in \citet{ChSh11}).

\begin{proof}
    \ignore{Without loss of generality, we assume that $a<0<b$ and  $w_0 = 0$; so, $p(0) = c_1$.
  We first prove \eqref{l4.1-1}.
  For $w\leq z$, define $H(w) = F(w)(1-F(z))-p(w)/c_1$. To prove \cref{l4.1-1}, noting that $f_z(w) \geq 0$, it suffices to show $\sup_{a < w < b} H_z(w) \leq 0$.  As $g(w)$ is non-decreasing, there is at most one solution to
  the equation $1-F(z)+g(w)/c_1 =0$ and hence $ \sup_{a < w \leq z} H(w) = \max( H(a) , H(z))$. Clearly,
  $H(a) = - p(a)/c_1\leq 0$ and  $H(z) =F(z)(1-F(z)) - p(z)/c_1  \leq   1- F(z) - p(z)/c_1$.
  If $z\leq 0$, then $H(z) \leq H(0)\leq 1 - p(0)/c_1 =0$; if $z>0$, consider $H_1(z) = 1-F(z) - p(z)/c_1$; then,  using a similar argument, $\sup_{0 < z < b} H_1(z) = \max( H_1(0), H_1(b)) \leq 0$. This proves
  $f_z(w) \leq 1/c_1$ for $w \le z$.
  Similarly, for $w>z$, we also have $f_z(w)\leq 1/c_1.$}

  Without loss of generality, we assume that $a < 0 < b$ and $w_0 = 0$; thus, $p(0) = c_1$.
For $w \leq z$, define $H_z(w) = F(w) (1 - F(z)) - p(w) / c_1$. %Then, $H_z'(w) = p(w) (1 - F(z) + g(w) / c_1)$.
To prove \eqref{l4.1-1}, noting that $f_z(w) \geq 0$,  it suffices to show that $\sup_{a < w < b} H_z(w) \leq 0$.
%We now consider the case when $w \leq z$.
As $g(w)$ is non-decreasing, by the fact that  $H_z'(w) = p(w) ( 1 - F(z) + g(w) / c_1)$, %there is at most one solution to the equation $1 - F(z) + g(w) / c_1 =  0$ and hence
\[
    \sup_{a < w \leq z} H_z(w) = \max \{ H_z(a), H_z(z) \}.
    \]
Clearly, $H_z(a) = - p(a)/c_1 \leq 0$. Now we prove $\sup_{a < z < b} H_z(z) \leq 0$. If $z \leq 0$, define $H_1(z) = F(z) - p(z)/c_1$ and thus $H_1'(z) = p(z) (1 + g(z)/c_1)$. Note that $g(z)\leq 0$ and $g(\cdot)$ is non-decreasing, then,
\[
    \sup_{a < z \leq 0} H_z(z) \leq \sup_{a < z \leq 0} H_1 (z) \leq \max \{ H_1 (a), H_1(0) \}  \leq 0.  %\leq \max \{ -p(a)/c_1, 1 - p(0)/c_1 \} \leq 0.
\]
Using a similar argument, we also have $\sup_{0 \leq z < b} H_z(z) \leq 0$. Therefore, $\sup_{a < z < b} H_z(z) \le 0$. This proves $\sup_{a < w \leq z} f_z(w) \leq 1/c_1$.
Similarly, we have $\sup_{z < w < b } f_z(w) \leq 1/c_1$.

A similar argument can be made for $w > z.$
This completes the proof of \eqref{l4.1-1}.
%define $H_1(z) = F(z) (1 - F(z)) - p(z)/c_1$, and thus $H_1'(z) = p(z) (1 - 2 F(z) + g(z)/c_1)$. Similarly, $1 - 2 F(z) + g(z)/c_1$ has at most one solution. Hence, $\sup_{a < z < b} H_1(z) \leq \max (H_1 (a), H_1(b)) \leq 0$.

  We next show that $g f_z$ is non-decreasing.
  For $w\leq z$, by \eqref{fz},
  \begin{eqnarray*}
    g(w)f_z(w) = {g(w) F(w) (1-F(z)) \over p(w)},%= (1-F(z))\psi(w),
  \end{eqnarray*}
  %where $\psi(w) = {g(w) F(w)  \over p(w)}$,
  and thus,
  \begin{eqnarray*}
    (g(w)f_z(w))' = (1-F(z))\big( g(w) + (g'(w)+ g^2(w)) F(w)/ p(w) \big).
  \end{eqnarray*}
%  It suffices to show $(g(w)f(w))'\geq 0$.
  Let $\tau(w) = {g(w) e^{-G(w)} \over g'(w) + g^2(w)}$. Then, by (A2),
  \begin{eqnarray*}
      -\tau'(w) e^{G(w)} = 1 - \Big( {2(g'(w))^2 - g''(w)g(w) \over (g'(w) + g^2(w))^2} \Big) \leq 1.
  \end{eqnarray*}
  %by (A3) .
  Hence,
  \begin{eqnarray*}
      e^{-G(w) } +  \tau'(w)  \geq 0
  \end{eqnarray*}
  and
  \begin{eqnarray*}
      0\le \int_a^w (\tau'(t) + e^{-G(t)} ) dt = \tau(w) +{1\over c_1} F(w) - \lim_{y \downarrow  a}\tau(y) .
  \end{eqnarray*}
  By condition (A3), $\lim_{y \downarrow  a}\tau(y)  = 0$ and hence  $\tau(w) +{1\over c_1} F(w)\geq 0$. This proves that
$(g(w)f_z(w))'\geq 0$ or $g(w)f_z(w)$ is non-decreasing for $w\leq z$. Similarly, one can prove that $g(w)f_z(w)$ is non-decreasing for $w\ge z$.
This proves \eqref{l4.1-4}.

  To prove (\ref{l4.1-3}),  by (A1), we have for $w \geq \max(z, 0)$,
  \begin{eqnarray*}
    g(w)f_z(w) & = & {F(z) g(w) \int_w^b p(t) dt \over p(w)} \nn \\
    & \leq & {F(z) \int_w^b e^{-G(t)}g(t) dt \over e^{-G(w)}} \leq F(z).
      %% & \leq &  {\int_w^b e^{-G(t)} dG(t) \over e^{-G(w)}}\\
      \end{eqnarray*}
     \ignore{ For $0 \leq w \leq z$, we have $g(w) \geq 0$, then,  by \eqref{l4.1-4},
      \begin{align*}
          0 \leq  g(w) f_z(w) \leq g(z) f_z(z) \leq 1.
      \end{align*}}
      Similarly, we have $g(w)f_z(w)\geq -(1-F(z))$ for $w\leq \min(0, z) $.
      Combining with \eqref{l4.1-4} yields
      \begin{equation}
      F(z) - 1 \leq g(w) f_z(w) \leq F(z) \label{l4.5-1}
      \end{equation}
      for all $w$. This proves (\ref{l4.1-3}).

  The inequality \eqref{l4.1-2} follows immediately from \eqref{eq1} and \eqref{l4.5-1}.

\end{proof}

%We are now ready to give the proof of \cref{thm-2}.
\begin{proof}[ Proof of \cref{thm-2}] Let $f=f_z$
be the solution to the Stein equation \eq{eq1}.
Since $(W,W')$ is an exchangeable pair,  by \eqref{con1}, we have
\begin{align*}
    %\label{thm-2-0}
    0 & = \ep{(W - W' ) (f(W) + f(W')) } \nl
      & = 2 \ep{(W - W' ) f(W)} - \ep{(W - W') (f(W ) - f(W'))}\nl
      & = 2 \lambda \ep{g(W)f(W)} + 2\lambda \ep{R f(W)} - \E \Big( \Delta \int_{-\Delta}^{0} f'(W + t) dt \Big),
\end{align*}
and hence, %By \eqref{thm-2-0}, we have
\begin{align*}
    %\label{thm-2-0.a}
    \ep{g(W)f(W)} = \frac{1}{2\lambda} \E \Bigl( \Delta \int_{-\Delta}^0 f'(W + t) dt  \Bigr) - \ep{R f(W)}.
\end{align*}
Thus,
\begin{eqnarray*}
\lefteqn{
\E (f'(W) - g(W) f(W)) } \\
& =& \E\Big (f'(W)\big( 1- { 1 \over 2\lambda} \E(\Delta^2 | W) \big)\Big) -
{ 1 \over 2 \lambda} \E \Big( \Delta \int_{- \Delta}^0 ( f'(W+t) - f'(W)) dt \Big) + \E( R f(W)).
\end{eqnarray*}
By \eqref{eq1}, \eq{l4.1-1} and \eq{l4.1-2},
\begin{eqnarray}
\lefteqn{
|\P( W \leq z) - \P(Y \leq z)| = | \E (f'(W) - g(W) f(W))| } \label{thm-2-1} \\
& \leq &  |I_1|  + 2 \E \Big|  1-{1\over 2\lambda} \conep{\Delta^2}{W}  \Big|  +
  {1\over c_1} \E|R|,\nn
\end{eqnarray}
where
$$
I_1 = {1\over 2\lambda} \epb{\Delta\int_{-\Delta}^0
  (f'(W+t) - f'(W) ) dt}   .
  $$
  Recalling  that $f$ is the solution to \eqref{eq1}, we have
  \begin{eqnarray}
  I_1 &=& {1\over 2\lambda} \epb{\Delta\int_{-\Delta}^0
  (g(W+t)f(W+t) - g(W)f(W) ) dt}  \label{thm-2-2}\\
  && + {1\over 2\lambda} \epb{\Delta\int_{-\Delta}^0
  (\indicator{ \left\{ W + t\leq z \right\}} - \indicator{ \left\{ W \leq z
  \right\}}) dt} . \nn
  %% \nn \\
  %%&:=& I_{1,1} + I_{1,2}. \nn
\end{eqnarray}
Noting that  $g(w)f(w)$ is non-decreasing by Lemma \ref{l4.1} and that the indicator function $\indicator{\{w \leq z\}}$ is non-increasing, we have
\beq
0  & \geq  &   \int_{-\Delta}^0 (g(W+t)f(W+t) - g(W)f(W)) dt \\
& \geq  & - \Delta   ( g(W)f(W) - g(W-\Delta) f(W-\Delta))
\eeq
and
$$
0  \leq   \int_{-\Delta}^0
  (\indicator{ \left\{ W + t\leq z \right\}} - \indicator{ \left\{ W \leq z
  \right\}}) dt
   \leq  \Delta  \Big(\indicator{ \left\{ W - \Delta \leq z \right\}} - \indicator{ \left\{ W \leq z   \right\} } \Big).
$$
  Therefore
  \begin{eqnarray}
  I_1 & \leq & {  1 \over 2 \lambda} E\Big(-  \Delta \indicator{ \{ \Delta <0\}}
 \Delta   ( g(W)f(W) - g(W-\Delta) f(W-\Delta)) \Big) \nn \\
 && + { 1 \over 2 \lambda} E\Big( \Delta \indicator{ \{ \Delta > 0\}}\Delta  \Big(\indicator{ \left\{ W - \Delta \leq z \right\}} - \indicator{ \left\{ W \leq z   \right\} } \Big)\Big).
 \label{thm-2-3}
\end{eqnarray}
Thus, for any $\Delta^*=\Delta^*(W, W')= \Delta^*(W',W) \geq |\Delta|$
\beqn
\lefteqn{
{  1 \over 2 \lambda} E\Big(-  \Delta \indicator{ \{ \Delta <0\}}
 \Delta   ( g(W)f(W) - g(W-\Delta) f(W-\Delta)) \Big) } \nn \\
 & \leq & {  1 \over 2 \lambda} E\Big( \Delta^* \indicator{ \{ \Delta <0\}}
 \Delta   ( g(W)f(W) - g(W') f(W') ) \Big) \nn \\
 & = & {  1 \over 2 \lambda} E\Big( \Delta^* \Delta ( \indicator{ \{ \Delta <0\}}
 + \indicator{ \{ \Delta > 0\}} ) g(W)f(W) \Big) \nn \\
 & =& {  1 \over 2 \lambda} E\Big( \Delta \, \Delta^*  g(W)f(W) \Big) \nn \\
& \leq & \frac{1}{ 2 \lambda} \E\left|  \conep{\Delta\Delta^*}{W}  \right|, \label{thm-2-3a}
\end{eqnarray}
where $\E( \Delta^* \Delta  \indicator{ \{ \Delta <0\}} g(W') f(W'))
= - \E( \Delta^* \Delta  \indicator{ \{ \Delta > 0\}} g(W) f(W))
$ because of the exchangeability of $W$ and $W'$ and  $|g(w)f(w)|\leq 1$ for all $w\in \R$.
Similarly, we have
\begin{equation}
{ 1 \over 2 \lambda} E\Big( \Delta \indicator{ \{ \Delta > 0\}}\Delta  \Big(\indicator{ \left\{ W - \Delta \leq z \right\}} - \indicator{ \left\{ W \leq z   \right\} } \Big)
\leq  \frac{1}{ 2 \lambda} \E\left|  \conep{\Delta\Delta^*}{W}  \right|. \label{thm-2-3b}
\end{equation}
Combining  \eq{thm-2-3}, \eq{thm-2-3a} and \eq{thm-2-3b} yields
\begin{equation}
I_1 \leq \frac{1}{  \lambda} \E\left|  \conep{\Delta\Delta^*}{W}  \right|.
\label{thm-2-3c}
\end{equation}

Following the same argument, we also have
\begin{equation}
I_1 \geq - \frac{1}{  \lambda} \E\left|  \conep{\Delta\Delta^*}{W}  \right|.
\label{thm-2-3d}
\end{equation}

This proves \eq{res1}, by \eq{thm-2-1}, \eq{thm-2-3c} and \eq{thm-2-3d}.
\end{proof}

\ignore{
$$
is non-positive and moreover,
\begin{eqnarray*}
\lefteqn{\left|  \int_{-\Delta}^0 (g(W+t)f(W+t) - g(W)f(W)) dt  \right|} \\ & \leq &  \Delta
  (g(W)f(W) - g(W-\Delta) f(W-\Delta)).
\end{eqnarray*}
Therefore, for any $\Delta^*=\Delta^*(W, W')= \Delta^*(W',W) \geq |\Delta|$,
\begin{eqnarray}
|I_{1,1}| & \leq &
 \frac{1}{2\lambda}\ep{|\Delta| \Delta (g(W)f(W) - g(W')f(W'))} \label{thm-2-3}\\
& \leq & \frac{1}{2\lambda}\ep{\Delta^*\Delta (g(W)f(W) - g(W')f(W'))}\nn \\
  &=& \frac{1}{\lambda}\ep{g(W)f(W)\Delta
  \Delta^* }\nn \\
  & \leq & \frac{1}{\lambda} \E\left|  \conep{\Delta\Delta^*}{W}  \right|, \nn
\end{eqnarray}
where $\E( \Delta \Delta^* g(W')f(W') ) = - \E( \Delta \Delta^* g(W)f(W) )$
because of the exchangeability of $W$ and $W'$ and  $|g(w)f(w)|\leq 1$ for all $w\in \R$.

As for $I_{1,2}$, the indicator  function $\indicator{\{t \leq z\}}$ is non-increasing. Following the same argument, we also have
\begin{equation}
|I_{1,2} | \leq \frac{1}{\lambda} \E\left|  \conep{\Delta\Delta^*}{W}  \right|. \label{thm-2-4}
\end{equation}

This proves \eq{res1}, by \eq{thm-2-1}--\eq{thm-2-4}.
\end{proof}

}

\section{Proofs of Theorems \ref{thma1}--\ref{thm:me}}
\label{sec:proof}

In this section, we give proofs for the  theorems in \cref{application}. The construction of an exchangeable pair is  described as follows.

Let $\eta_1, \cdots, \eta_n$ be a sequence of random variables and $W=h(\eta_1, \cdots, \eta_n)$.
For each $1 \leq i \leq n$, let $\eta_i'$ have the  conditional distribution of $\eta_i$  given $\{\eta_j, 1 \leq j \leq n,  j \not = i\}$, also, $\eta_i'$ is conditionally independent of $\eta_i$ given $ \left\{ \eta_j, 1 \leq j \leq n , j \neq i \right\}$.
 Let  $I$ be a random
 index uniformly distributed over $\{1, \cdots, n\}$ independent of $ \left\{ \eta_i, \eta_i', 1 \leq i \leq n \right\}$. Set
$$ W'= h(\eta_1, \cdots, \eta_{I-1}, \eta_I' , \eta_{I+1}, \cdots, \eta_n). $$ Then, $(W, W')$ is an exchangeable pair. In particular, when
$\eta_i, 1 \leq i \leq n$ are independent, one can let $\{\eta'_i, 1 \leq i \leq n \}$ be an independent copy of $\{\eta_i, 1 \leq i \leq n\}$.
This sampling procedure is also called the Gibbs sampler.

\subsection{Proof of \cref{thma1}}
%\begin{proof}[Proof of \cref{thma1}]
Let $\X=\sigma(X_1,\cdots,X_n)$, and $(X_1',X_2',\cdots, X_n')$ be an  independent copy
  of $(X_1,X_2,\cdots, X_n)$. Let $I$ be a random index uniformly distributed over
  $\left\{ 1,\cdots,n \right\}$ independent of any other random variable.
  Write $W_n=h(X_1,\cdots,X_n)$ and define $W_n' = h(X_1,\cdots,X_I', \cdots, X_n)$.
  Then,
 $(W_n,W_n')$ is an exchangeable pair.
 %`Without loss of generality, assume that $\sum_{i,j} a_{ij}^2 = 1$.
 It is easy to see that
 \begin{equation*}
     \Delta = W_n - W_n' = \frac{2}{\sigma_n} \sum_{j\neq I} a_{jI}X_j (X_I - X_I'),  \label{thm3.1-1}
  \end{equation*}
and
  \begin{eqnarray*}
      \conep{\Delta}{\X} &=& \frac{2}{\sigma_n} \sum_{i=1}^n \sum_{j\neq i}
    \conep{a_{ji}X_j (X_i-X_i')}{\X} \label{thm3.1-2}\\
    &=& {2\over n} W_n. \nn
  \end{eqnarray*}
  As such,  condition \eqref{con0} holds with  $\lambda = 2/n$ and $R=0$.
  Also,
  \begin{eqnarray*}
      \conep{\Delta^2}{\X} &=& \frac{4 }{ n \sigma_n^2} \sum_{i=1}^n \E\biggl(\Bigl(    \sum_{j\neq i} a_{ji} X_j (X_i-X_i') \Bigr)^2 \biggm| {\X} \biggr)\\
                           &=& \frac{4 }{n \sigma_n^2} \sum_{i=1}^n (X_i^2 + 1)\Bigl(  \sum_{j=1}^n a_{ij}X_j
    \Bigr)^2,
  \end{eqnarray*}
  and
  \begin{eqnarray*}
      {1\over 2\lambda} \conep{\Delta^2}{\X} &=& \frac{1}{\sigma_n^2 }  \sum_{i=1}^n (X_i^2 + 1)\Bigl(  \sum_{j=1}^n a_{ij}X_j
    \Bigr)^2.\label{qua-1}
  \end{eqnarray*}
  Note that by the assumptions $\sigma_n^2 = 2 \sum_{i,j} a_{ij}^2$ and $a_{ii}=0$,
  \[
  \E \Big( {1\over 2\lambda} \conep{\Delta^2}{\X}  \Big) = 1.
  \]
  Then,
  \begin{eqnarray*}
      \E\bigl|1- {1\over 2\lambda}\conep{\Delta^2}{W_n} \bigr|^2 \leq   \Var\Bigl(  \frac{1}{\sigma_n^2} \sum_{i=1}^n (X_i^2 + 1)\Bigl(  \sum_{j=1}^n a_{ij}X_j
    \Bigr)^2 \Bigr).
    %C \Big( \sqrt{ \sum_{i}\Big(\sum_j a_{ij}^2 \Big)^2}+ \sqrt{\sum_{i,j}\Big(\sum_{k} a_{ik}a_{jk} \Big)^2 }    \Big)\label{qrc1}
  \end{eqnarray*}
  %where $C$ is a constant depending only on $\E X_1^4.$
Observe that
  \begin{eqnarray}
    \lefteqn{\Var\Bigl( \sum_{i=1}^n (X_i^2 + 1)\Bigl(  \sum_{j=1}^n a_{ij}X_j
    \Bigr)^2 \Bigr) }\nonumber\\
    & = & \sum_{i=1}^n \Var\Big( (X_i^2 + 1)\Bigl(  \sum_{j=1}^n a_{ij}X_j
    \Bigr)^2   \Big) \nonumber\\
    && + \sum_{i\neq i'} \Cov\Big( (X_i^2 + 1)\Bigl(  \sum_{j=1}^n a_{ij}X_j
    \Bigr)^2 , (X_{i'}^2 + 1)\Bigl(  \sum_{k=1}^n a_{i'k}X_k
    \Bigr)^2    \Big).\label{qua-2}
  \end{eqnarray}

For the first term, recalling that  $a_{ii}=0$
for all $1 \leq i \le n $, we have
  \begin{eqnarray}
    \lefteqn{ \sum_{i=1}^n \Var\Big( (X_i^2 + 1)\Bigl(  \sum_{j=1}^n a_{ij}X_j
    \Bigr)^2   \Big) } \label{thm3.1-3} \\
    & \leq  & \sum_{i=1}^n \E  (X_i^2 + 1)^2 \E \Bigl(  \sum_{j=1}^n a_{ij}X_j
    \Bigr)^4 \nn \\
           & \leq & C  \sum_{i=1}^n (\E(X_1^4) +
           1 )  \E(X_1^4) \Big( \sum_{j=1}^n
             a_{ij}^4 + \Big( \sum_{j=1}^n
           a_{ij}^2 \Big)^2\Big)
            \nn \\
            & \leq & C (\E(X_1^4) )^2 \sum_{i=1}^n \Big( \sum_{j=1}^n a_{ij}^2  \Big)^2 ,  \nn
  \end{eqnarray}
  where $C$ is an absolute constant.
%where in the third inequality we used the
%Rosenthal inequality and $C$ is an absolute
%constant.
  To bound the second term of \eqref{qua-2}, for
  any $i\neq k$,
  define
  \begin{eqnarray*}
      M_{i} &=&  (X_i^2 + 1 )
      \Big( \sum_{j = 1}^n a_{ij} X_j  \Big) ^2,
      \\
      M_{i}^{(k)} &=& (X_i^2 + 1 ) \Big(
      \sum_{j \neq k}^n a_{ij} X_j  \Big) ^2.
  \end{eqnarray*}
  For the second term of \eqref{qua-2}, for any $i\neq i'$, we have
  \begin{eqnarray}\label{qua-3}
    \lefteqn{ \Cov\Big( (X_i^2 + 1)\Bigl(  \sum_{j=1}^n a_{ij}X_j
    \Bigr)^2 , (X_{i'}^2 + 1)\Bigl(  \sum_{k=1}^n a_{i'k}X_k
    \Bigr)^2    \Big)  }\nonumber\\
    &=& \Cov(M_i, M_{i'})\nonumber\\
    &=& \Cov( M_{i}^{(i')} , M_{i'}   ) +
    \Cov (M_i, M_{i'}^{(i)}  ) \nonumber\\
    && - \Cov(
    M_{i}^{(i')} , M_{i'}^{(i)} ) + \Cov
    (M_i - M_{i}^{(i')} , M_{i'}-  M_{i'}^{(i)}).
    \end{eqnarray}
%Let $\F_i=\sigma\left\{ X_i \right\}$. Given
%    $\F_i$, we know $ M_{i}^{(i')} $ and $M_{i'}$ are conditionally independent, and we know
%    \begin{eqnarray}\label{thm3.1-a}
%        \Cov( M_{i}^{(i')} , M_{i'}) &=& \Cov\Big( \conep{ M_{i}^{(i')} }{\F_i} , \conep{M_i'}{\F_i}\Big).
%    \end{eqnarray}
%Observe that
%\begin{eqnarray*}
%    \conep{ M_{i}^{(i')}  }{\F_i} &=&  (X_i^2 + 1) \ep{\sum_{j \neq i'} a_{ij} X_j}^2\\
%    &=& (X_i^2 + 1) \Big( \sum_{j \neq i'} a_{ij}^2 \Big) ,
%\end{eqnarray*}
%and
%\begin{eqnarray*}
%    \conep{M_{i'}}{\F_i}&=& \conep{ (X_{i'}^2 + 1) \Big( \sum_{j\neq i} a_{i'j}X_j + a_{i'i} X_j  \Big) ^2  }{\F_i}\\
%    &=& \ep{(X_{i'}^2 + 1) \Big( \sum_{j\neq i} a_{i'j}X_j \Big)^2 } + 2 a_{i'i}^2 X_i^2 .
%\end{eqnarray*}
%Therefore, by \eqref{thm3.1-a}
%\begin{eqnarray*}
%    |\Cov( M_{i}^{(i')} , M_{i'})| &=& 2 a_{i'i}^2 \Big( \sum_{j \neq i'} a_{ij}^2 \Big) \Var(X_i^2)\\
%    & \leq & 2 \ep{X_1^4} a_{i'i}^2 \Big( \sum_{j\neq i'} a_{ij}^2 \Big) .
%\end{eqnarray*}
%Similarly,
%\begin{eqnarray*}
%    |\Cov(M_{i}, M_{i'}^{(i)} )| \leq 2 \ep{X_1^4} a_{ii'}^2 \Big( \sum_{j\neq i} a_{i'j}^2 \Big) .
%\end{eqnarray*}
%

  Given  $\F_{ii'} := \sigma\{X_{j}, j\neq i,i'\}$, random variables  $ M_{i}^{(i')} $ and $ M_{i'}^{(i)} $ %$(X_i^2 + 1) \Bigl(  \sum_{j\neq i'}^n a_{ij}X_j \Bigr)^2$ and $(X_{i'}^2 + 1)\Big( \Bigl(  \sum_{k \neq i }^n a_{i'k}X_k
    %\Bigr)^2$
  are independent. Thus,
    \begin{eqnarray}
        \lefteqn{  \Cov( M_{i}^{(i')} , M_{i'}^{(i)} )  }\nonumber\\
        & =& \Cov\Big( \E\Big( (X_i^2 + 1) \Bigl(  \sum_{j\neq i'}^n a_{ij}X_j
    \Bigr)^2 \ \Big\vert \F_{ii'}\Big), \E\Big( (X_{i'}^2 + 1) \Bigl(  \sum_{k\neq i}^n a_{i'j}X_k
    \Bigr)^2 \ \Big\vert \F_{ii'}\Big)\Big) \nn \\
    & =& 4 \Cov\Big( \Bigl(  \sum_{j\neq i'}^n a_{ij}X_j    \Bigr)^2, \Bigl(  \sum_{k\neq i}^n a_{i'k}X_k    \Bigr)^2
     \Big) \nn \\
    & \leq &  C
    \sum_{j=1}^n a_{ij}^2a_{i'j}^2 \E(X_1^4) + C  \Big( \sum_{k=1}^n a_{ik}a_{i'k} \Big) ^2.  \nn
   %\sum_{j_1 \neq j_2 \neq i', k_1\neq k_2 \neq
   %i,\atop \{j_1, j_2\} = \{ k_1, k_2\} }
   %a_{ij_1} a_{ij_2} a_{i'k_1} a_{i'k_2} . \nn
\end{eqnarray}
Similar arguments hold  for other terms of \eqref{qua-3}. Hence,
  \begin{eqnarray}
    \lefteqn{ \sum_{i\neq i'}\Cov\Big( (X_i^2 + 1)\Bigl(  \sum_{j=1}^n a_{ij}X_j
    \Bigr)^2 , (X_{i'}^2 + 1)\Bigl(  \sum_{k=1}^n a_{i'k}X_k
    \Bigr)^2    \Big)  }\label{thm3.1-4}\\
    & \leq &  C \E(X_1^4)^2 \Big( \sum_{i=1}^n \Big( \sum_{j=1}^n a_{ij}^2 \Big)^2 + \sum_{1\leq i,j\leq n} \Big( \sum_{k=1}^n a_{ik}a_{jk}\Big)^2 \Big). \nn
  \end{eqnarray}
It follows from \eq{qua-2}, \eq{thm3.1-3} and \eq{thm3.1-4} that
  \begin{eqnarray}
    \nonumber\lefteqn{\E\bigl|1- {1\over
    2\lambda}\conep{\Delta^2}{W_n} \bigr| } % \label{qua-4}
    \\ &  \le  & \
    C \sigma_n^{-2} \E(X_1^4)\Big( \sqrt{ \sum_{i}\Big(\sum_j
     a_{ij}^2 \Big)^2}+
   \sqrt{\sum_{i,j}\Big(\sum_{k} a_{ik}a_{jk}
 \Big)^2 }    \Big)\label{qrc1}.  % \nn
  \end{eqnarray}
Finally, it is sufficient to estimate the bound of $\E|\conep{\Delta|\Delta| }{W_n} |/\lambda$. In fact,
  \begin{eqnarray*}
      \frac{1}{\lambda}\lefteqn{\conep{\Delta|\Delta| }{\X}}\\
    &=& \frac{2 }{\sigma_n^2} \sum_{i=1}^n \conepbb{ \Big(\sum_{j} a_{ij} X_j (X_i - X_i')   \Big) \Big|\sum_j a_{ij} X_j (X_i - X_i')   \Big|}{\X}\\
    & = & \frac{2 }{ \sigma_n^2} \sum_{i=1}^n \Big( \sum_j
    a_{ij}X_j \Big) \Big| \sum_j a_{ij}X_j \Big| B_i,
  \end{eqnarray*}
  where $B_i =
  \ep{(X_i-X_i')|X_i-X_i'|\big \vert X_i}$.

  For $i \neq i'$, define
  \begin{eqnarray*}
      K_i &=&  \Big( \sum_j a_{i j}X_{j} \Big) \Big| \sum_j a_{i j }
      X_j  \Big| B_{i} ,  \\
    K_i^{(i')} &=&  \Big( \sum_{j\neq i'} a_{i j}
    X_{j}\Big) \Big|
    \sum_{j\neq i'}  a_{i j }
    X_j  \Big| B_{i}
  \end{eqnarray*}
  and thus,
  \begin{eqnarray*}
      \Var( (1/\lambda) \conep{\Delta|\Delta|}{\X} ) &=&
    \frac{4}{\sigma_n^4 } \sum_{i=1}^n \Var(K_i) +
    \frac{4}{\sigma_n^4} \sum_{i\neq i'} \Cov(K_i,
    K_{i'}).
  \end{eqnarray*}

  Similar to \eqref{thm3.1-3}, we have
  \begin{eqnarray*}
    \sum_{i=1}^n \Var(K_i) \leq C (\ep{X_1^4})^2
    \sum_{i=1}^n \Big( \sum_{j=1}^n
    a_{ij}^2\Big)^2.
  \end{eqnarray*}

  Recalling  the definition of $\F_{ii'}$, given that  $\F_{ii'}$, we have $K_i^{(i')}$ and
  $K_{i'}^{(i)}$ are conditionally independent, and  thus
  \begin{eqnarray*}
      \Cov( K_{i}^{(i')}, K_{i'}^{(i)} | \F_{ii'}  ) = 0.
  \end{eqnarray*}
  Moreover,
  \begin{eqnarray*}
      \conep{K_{i}^{(i')}}{\F_{ii'}} = \Big( \sum_{j\neq
    i'} a_{i j} \Big) \Big|
    \sum_{j\neq i'}  a_{i j }
    X_j  \Big| \E( B_{i}) = 0
  \end{eqnarray*}
  because $\ep{B_i} = 0$. This proves
  $\Cov(K_{i}^{(i')}, K_{i'}^{(i)}) =0$.
  Similarly, we have
  $\Cov(K_{i}^{(i')}, K_{i'}) = 0$ and $\Cov(K_i,
  K_{i'}^{(i)}) = 0$.
  Therefore,
  \begin{eqnarray*}
      | \Cov(K_i,K_{i'})| &=& \E\big|(K_i -
    K_{i}^{(i')})(K_{i'} - K_{i'}^{(i)})  \big| \\
    &\leq &
    \frac{1}{2} \E(K_i - K_{i}^{(i')})^2 +
    \frac{1}{2}\ep{K_{i'} - K_{i'}^{(i)}}^2.
  \end{eqnarray*}
  Observe that
  \begin{eqnarray*}
      |K_{i} - K_{i}^{(i')} | \leq |B_i| \Big( 2 \Big|
    a_{ii'} X_{i'} \sum_{j\neq i'} a_{ij} X_j
\Big| + a_{ii'}^2 X_{i'}^2 \Big),
  \end{eqnarray*}
  thus,
  \begin{eqnarray*}
      \E(K_i - K_{i}^{(i')})^2 & \leq & C \E(B_i)^2 \Big(
      a_{ii'}^2 \sum_{j} a_{ij}^2 + a_{ii'}^4
      \ep{X_1^4}
    \Big) \\
    & \le &  C (\ep{X_1^4})^2 \Big(
      a_{ii'}^2 \sum_{j} a_{ij}^2 + a_{ii'}^4
     \Big).
  \end{eqnarray*}
  A similar result is true for $\ep{K_{i'} -
  K_{i'}^{(i)}}^2$.
  Combining the inequlities, we have
  \begin{eqnarray*}
      \Var( \conep{\Delta|\Delta|}{\X}  / \lambda ) \leq C \sigma_n^{-4}
    (\ep{X_1^4})^2 \sum_{i=1}^n  \Big(
    \sum_{j = 1}^n a_{ij}^2 \Big)^2.
  \end{eqnarray*}

  By the Cauchy inequality, we have
  \begin{eqnarray}
  \lefteqn{
    {1\over \lambda} \E\big|\conep{\Delta|\Delta| }{W_n}  \big|} \label{qrc2} \\
    & \leq & C \sigma_n^{-2}\E(X_1^4)\Big( \sqrt{ \sum_{i}\Big(\sum_j a_{ij}^2 \Big)^2}+ \sqrt{\sum_{i,j}\Big(\sum_{k} a_{ik}a_{jk} \Big)^2 }    \Big) . \nn
  \end{eqnarray}
  %(This need to be finished.)
This  completes the proof of  \cref{thma1}
by \eqref{qrc1} and \eqref{qrc2}.
%\end{proof}

\def\s{\bar{X}}

\subsection{Proof of \cref{cwmthm1}}
Recall that $S_n = \sum_{i=1}^n X_i$.
Let $\X=\sigma(X_1,\cdots,X_n)$. % $\s=S_n/n$ and
%$\s_i = {1\over n}(S_n - X_i)$.
We first construct an exchangeable pair $(S_n,S_n')$ as follows. For each $1\leq i\leq n$, given $\{X_j, j\neq i\}$, let $X_i'$ be conditionally independent of $X_i$ with the same conditional distribution of $X_i$. Let $I$ be a random index uniformly distributed over $\{1,\cdots,n\}$ independent of any other random variable. Define $S_n' = S_n - X_I + X_I'$;   then, $(S_n,S_n')$ is an exchangeable pair.
%ombining all the inequalities above, we have the desired result.

The proof of \cref{cwmthm1} is based on the following propositions.
Let $\s = S_n / n$.

\begin{proposition}
    Under the assumptions in \cref{cwmthm1},
    for $\beta = 1$, we have
    \begin{eqnarray}
        \label{lc1-2}
        \conep{S_n - S_n'}{\X} = \frac{H^{(2k)}(0) }{(2k-1)!} \s^{2k-1} + R_1,
    \end{eqnarray}
    where $\E |R_1| \leq C n^{-1}$ with the constant $C$ depending only on $b_0,b_1,b_3$ and $k$.

    For $0 < \beta < 1$,  we have
    \begin{eqnarray}\label{lc1-1}
        \conep{S_n - S_n'}{\X} = (1-\beta) \bar{X} + R_2,
    \end{eqnarray}
    where $\E|R_2| \leq C n^{-1}$ and $C$ depends only on $\beta$ and $b$.
    \label{l-c1}
\end{proposition}

\begin{proposition}
    Under the assumptions in \cref{cwmthm1}, % and $\beta = 1$,
    we have
    \begin{eqnarray}\label{lc2-1}
        \E \big| \conep{(S_n-S_n')^2}{\X} - 2  \big| \leq C n^{-1/2}, \text{ for } 0 < \beta < 1,
    \end{eqnarray}
    and
    \begin{eqnarray}
        \label{lc2-2}
        \E \big| \conep{(S_n-S_n')^2}{\X} - 2  \big| \leq C n^{-1/2k}, \text{ for } \beta=1.
    \end{eqnarray}
    \label{l-c2}
\end{proposition}
%\cref{lc2-1,lc2-2}.

\begin{proposition}
    Under the assumptions in \cref{cwmthm1}, we have for $0 < \beta  \leq 1$,
        \begin{eqnarray}\label{lc3-1}
        \E \big| \conep{(S_n-S_n')|S_n-S_n'|}{\X}  \big| \leq C n^{-1/2}. %,  \text{ for } 0 < \beta \leq  1.
    \end{eqnarray}
   % and
   % \begin{eqnarray}
   %     \label{lc3-2}
   %     \E \big| \conep{(S_n-S_n')|S_n-S_n'|}{\X}   \big| \leq C n^{-1/2k},  \text{ for } \beta=1.
   % \end{eqnarray}
    \label{l-c3}
\end{proposition}

We now continue to prove \cref{cwmthm1}.
\begin{enumerate}[(i)]

    \item
When $0 < \beta <1$, define $W_n = S_n/\sqrt{n}$ and $W_n' = S_n'/\sqrt{n}$.  Then, $(W_n,W_n')$ is an exchangeable pair, and  by \eqref{lc1-1} in \cref{l-c1},%\eqref{cwma1},
we have
\begin{eqnarray*}
    \conep{W_n-W_n'}{W_n} = {1 \over n}((1-\beta)W_n + \sqrt{n}R_2 ),
\end{eqnarray*}
where $\E|R_2| \leq C n^{-1}$.
%and $\sqrt{n}\E|R_1|\le C_{b,\beta} n^{-1/2}$.
Moreover, taking $\lambda = 1/n$,  by  \eqref{lc2-1} and \eqref{lc3-1}, we have
\begin{eqnarray*}
  \E\Big| 1 - {1\over 2\lambda} \conep{(W_n-W_n')^2}{W_n}\Big| \leq C n^{-1/2}
\end{eqnarray*}
and
\begin{eqnarray*}
  \E\Big|{1\over 2\lambda} \conep{(W_n-W_n')|W_n-W_n'|}{W_n}\Big| \leq C n^{-1/2}
\end{eqnarray*}
for some constant $C$.
This proves \eqref{cwmt1-1} by \cref{thm-2} with $g(w) = (1-\beta)w$.
%Then by \cref{thm-2} and taking $g(w) = (1-\beta)w$, we get the first result in \cref{cwmthm1}.
\item
When $\beta = 1$, define $W_n = n^{-1+\frac{1}{2k}} S_n$ and $W_n'=n^{-1+\frac{1}{2k}} S_n'$. Then, $(W_n,W_n')$ is an exchangeable pair, and by \eqref{lc1-2},
\begin{eqnarray*}
    \conep{W_n-W_n'}{W_n} = n^{-2+1/k} \Big( { {H^{2k}(0)\over (2k-1)!} W_n^{2k-1} } + n^{-1+\frac{1}{2k}} R_1   \Big),
\end{eqnarray*}
where $n^{-1+\frac{1}{2k}} \E|R_1|\leq C n^{-\frac{1}{2k}}$. Taking $\lambda = n^{-2+1/k}$ and by \eqref{lc2-2} and \eqref{lc3-1}, we have
\begin{eqnarray*}
  \E\Big| 1 - {1\over 2\lambda}
  \conep{(W_n-W_n')^2}{W_n}\Big| \leq C n^{-\frac{1}{2k}},
\end{eqnarray*}
and
\begin{eqnarray*}
  \E\Big|{1\over 2\lambda}
  \conep{(W_n-W_n')|W_n-W_n'|}{W_n}\Big| \leq C n^{-1/2}.
\end{eqnarray*}

This completes the proof of \eqref{cwmt1-2} by  \cref{thm-2} with  $g(w) = { {H^{2k}(0)\over (2k-1)!} w^{2k-1} } $. % we have the second result in \cref{cwmthm1} holds.
\end{enumerate}
%We must  prove \cref{l-c1,l-c2,l-c3} %-\cref{l-c3}
%in the main proof.
%Let
%$(X_1,\dots,X_n)$ be a random vector following
%the joint distribution $P_{n,\beta}(\x)$. Define
%\begin{eqnarray*}
%  S_n &=& X_1+\dots+ X_n,\\
%  s &=&  S_n/n,\\
%  s_i &=&  \frac{ 1}{n} (S_n-X_i).
%\end{eqnarray*}

To prove \cref{l-c1,l-c2,l-c3}, we need to prove some preliminary lemmas.

In what follows, we let $\xi, \xi_1,\xi_2,\dots$
be independent and identically distributed random
variables with probability measure $\rho$
satisfying
\eqref{cwmc0}, and \eqref{c1-00} or
\eqref{c1-01}.
%It suffices to prove the following lemmas.

\begin{lemma}\label{cwmlem1}
%Let $\xi_1,\xi_2,...$ be a sequence of
%independent random
%variables with  the distribution $\rho$.
%satisfying \eqref{c1-00}.
For any $z>0$,  under
          \eqref{c1-00}, we have
  \begin{eqnarray}
      \P(|\xi_1+\cdots+\xi_n| > z) \leq 2\exp\Big(- {bz^2\over 2n} \Big),  \text{ for } 0 < \beta < 1. \label{cwmeq3}
  \end{eqnarray}
  %\item
  Under \eqref{c1-01}, and for $\beta=1$,
      \begin{eqnarray}\label{cl-1}
    \lefteqn{\P(| \xi_1 + \cdots + \xi_n|  > z)}\nonumber \\
    & \leq &
    \begin{cases}
      2 \expb{ - \frac{z^2}{2n} -
      \frac{b_1 z^{2k}}{n^{2k-1}}}, & 0 < z \leq b_0
      n,\\
      2 \expb{- \frac{b_2 z^2 }{ 2n}}, & z > b_0n.
    \end{cases}
  \end{eqnarray}
%\end{itemize}
\end{lemma}
\begin{proof}
    \eqref{cwmeq3} follows easily from \eqref{c1-00} and Chebyshev's inequality.

    As for \eqref{cl-1}, when $0 < z \leq b_0n$, set $t = z/n$.
  By the  Chebyshev inequality, we have
  \begin{eqnarray*}
    \P( \xi_1 + \cdots + \xi_n > z) & \leq & e
    ^{-tz} \E e^{t ( \xi_1 + \cdots + \xi_n)}\\
    &=& e^{-tz}  \left( \E e^{ t \xi} \right)^n
    \\
    & \le & e^{-tz} \expb{ \frac{nt^2}{2} - n b_1
    t^{2k}} \\
    &=&  \expb{ - \frac{z^2}{2n} -
    \frac{b_1 z^{2k}}{n^{2k-1}}}.
  \end{eqnarray*}
  Similarly,
  \begin{eqnarray*}
    \P( \xi_1 + \cdots + \xi_n < -z ) \leq  \expb{ - \frac{z^2}{2n} -
    \frac{b_1 z^{2k}}{n^{2k-1}}},
  \end{eqnarray*}
  and hence
  \begin{eqnarray*}
    \P( | \xi_1 + \cdots + \xi_n | >z ) \leq 2  \expb{ - \frac{z^2}{2n} -
    \frac{b_1 z^{2k}}{n^{2k-1}}}.
  \end{eqnarray*}

  A similar argument for $ z > b_0 n$ completes
  the proof of \eqref{cl-1}.
\end{proof}

%Now, we can move on to prove the following result.
\begin{lemma}
  Under condition \eqref{c1-01} and for $\beta = 1$, we have
  \begin{eqnarray}\label{cl2-1}
   cn^{\frac{1}{2} - \frac{1}{2k}} \leq  \E \expb{\frac{1}{2n}( \xi_1 + \cdots +
   \xi_n)^2}  \leq C n^{\frac{1}{2} -
   \frac{1}{2k}},
 \end{eqnarray}
where $c$ and $C$ are constants such that $0 < c < C < \infty$.
 Under condition \eqref{c1-00}, for $0 < \beta < 1$, we have
 \begin{eqnarray}\label{cl2-0}
     1 \leq  \E \exp \Big( \frac{\beta}{2n}(\xi_1+\dots + \xi_n)^2 \Big) \leq C,
 \end{eqnarray}
where $C> 1$ is a finite constant.

  \label{cl-2}
\end{lemma}
\begin{proof}
  Noting  that
  \begin{eqnarray*}
    e^{x^2/2} = \frac{1}{\sqrt{2\pi}}
    \int_{-\infty}^\infty e^{t x - t^2 /2  } dt,
  \end{eqnarray*}
we have
\begin{eqnarray}\label{cl2-a1}
    \lefteqn{ \sqrt{2\pi} \E \expb{\frac{1}{2n}( \xi_1 + \cdots +
   \xi_n)^2}  }\nonumber\\
   & = &
   \int_{-\infty}^{\infty} \E \expb{
   \frac{t}{\sqrt{n}}( \xi_1 + \cdots + \xi_n) -
 \frac{t^2}{2}} dt\nonumber\\
 & \le &  \int_{|t| \leq b_0 \sqrt{n}}
 e^{- b_1 t^{2k}/ n^{k-1}} dt  + \int_{ | t| >
 b_0 \sqrt{n}} e^{ - \frac{t^2}{2} ( 1 -
 \frac{1}{b_2} )} dt\nonumber\\
 & \leq & C n^{\frac{1}{2} -
 \frac{1}{2k}}
  \end{eqnarray}
  for some constant $C$.

  For the lower bound of  $\E e^{\frac{1}{2n}(\sum_{i=1}^n \xi_i)^2}$, as $\rho$ is of type $k$ with strength
$\lambda_\rho$, then by  the Taylor expansion, for $|t| \leq b_0$,
\begin{eqnarray*}
  \left| \frac{t^2}{2} - \log \E e^{t\xi} \right|
  \leq C_\lambda t^{2k},
\end{eqnarray*}
where $C_\lambda = \lambda_\rho + b_0 \sup_{|t|\leq b_0 } | H^{(2k + 1)} (t)|$ is a constant. Thus, for $|t| \leq
b_0$,
\begin{eqnarray*}
  \E e^{t\xi} \geq \expb{ \frac{t^2}{2} - C_\lambda
  t^{2k} }.
\end{eqnarray*}

Similar to \eqref{cl2-a1}, we have
\begin{eqnarray*}
    \sqrt{2\pi} \E e^{\frac{1}{2n} \big( \sum_{i=1}^n \xi_i \big) ^2 }  & \geq &  \E \int_{|t| \leq b_0\sqrt{n}} e^{\frac{t}{\sqrt{n}} (\xi_1+\dots+\xi_n) - \frac{t^2}{2}} dt \nonumber\\
    & \geq & c n^{\frac{1}{2} - \frac{1}{2k}}.
\end{eqnarray*}
This proves \eqref{cl2-1}.

Under condition \eqref{c1-00} and similar to \eqref{cl2-a1}, we have
\begin{eqnarray*}
    \E \exp \Big( \frac{\beta}{2n}(\xi_1+\dots + \xi_n)^2 \Big)  \leq C,
\end{eqnarray*}
and by the Jensen inequality,
\begin{eqnarray*}
    \E e^{\frac{\beta}{2n} \big( \sum_{i=1}^n \xi_i \big)^2 } & \geq &  e^{\frac{\beta}{2n} \ep{(\xi_1+\dots+ \xi_n)^2}} \nonumber\\
    &\geq & 1.
\end{eqnarray*}

This completes the proof of \eqref{cl2-0}.
\end{proof}

Let $X=(X_1,\cdots,X_n)$ be a random vector
following the Curie-Weiss distribution satisfying
\eqref{cwmpdf}. We have the following inequalities.

\begin{lemma}
    Under condition
    \eqref{c1-01}, we have
 \begin{eqnarray}
   \E \left( \frac{X_1+ \cdots +
   X_n}{n^{1-\frac{1}{2k}}} \right)^{2k} \leq C,
   \quad \beta=1, \label{cl3-1}
 \end{eqnarray}
 and under condition \eqref{c1-00}, we have
 \begin{eqnarray}
   \E \left( \frac{X_1+ \cdots +
   X_n}{\sqrt{n}} \right)^{2} \leq C,
   \quad
   0 < \beta < 1 . \label{cl3-2}
 \end{eqnarray}
  \label{cl-3}
\end{lemma}

\begin{proof}
    Let $M_n = \frac{1}{\sqrt{n}}(\xi_1+\dots+\xi_n)$ and $Z_n = \E e^{\frac{1}{2}M_n^2}$. For $\beta=1$ and when \eqref{c1-01} holds,
    by  \eqref{cwmpdf}, we have
    \begin{eqnarray}
        \ep{S_n^{2k}} &=&  \frac{n^{k}}{Z_n} \ep{ M_n^{2k} e^{\frac{1}{2} M_n^2} }\nonumber\\
        &=&\frac{ n^{k}}{Z_n} \int_{0}^{\infty} \big(2k x^{2k-1} + \frac{1}{2} x^{2k+1}\big) e^{\frac{1}{2}x^2} \P(|M_n| \geq x) dx \nonumber\\
        &=&\frac{ n^{k}}{Z_n} (I_1 + I_2), \label{cl3-01}
    \end{eqnarray}
    where
    \begin{eqnarray*}
        I_1 &=&  \int_0^{b_0 \sqrt{n}}  \big(2k x^{2k-1} + \frac{1}{2} x^{2k+1}\big) e^{\frac{1}{2}x^2} \P(|M_n| \geq x) dx , \nonumber\\
        I_2 &=& \int_{b_0 \sqrt{n}}^\infty  \big(2k x^{2k-1} + \frac{1}{2} x^{2k+1}\big) e^{\frac{1}{2}x^2} \P(|M_n| \geq x) dx .
    \end{eqnarray*}
    For $I_1$, letting $D_n = [b_0 \sqrt{n}]+1$, where $[a]$ is the integer part of $a$, we have
    \begin{eqnarray*}
        I_1 &\leq & \sum_{j=0}^{D_n} \int_j^{j+1}  \big(2k x^{2k-1} + \frac{1}{2} x^{2k+1}\big) e^{\frac{1}{2}x^2} \P(|M_n| \geq x) dx \nonumber\nonumber\\
        & \leq & C \Big(  1 + \sum_{j=1}^{D_n} j^{2k+1} \int_j^{j+1} e^{\frac{1}{2} x^2 - jx + jx } \P(|M_n| \geq x) dx \Big)\nonumber\\
        & \leq & C \Big(  1 + \sum_{j=1}^{D_n} j^{2k+1} e^{-\frac{j^2}{2}} \int_j^{j+1} e^{jx} \P(| M_n|\geq x) dx \Big)\nonumber\\
        & \leq &  C \Big( 1 + \sum_{j=1}^{D_n} j^{2k+1} e^{-\frac{j^2}{2}} \int_j^{j+1} e^{jx - \frac{x^2}{2} - \frac{b_1 x^{2k}}{n^{k-1}}}dx \Big)\quad \text{by \eqref{cl-1}} \nonumber\nonumber\\
        & \leq & C \Big( 1 + \sum_{j=1}^{D_n} j^{2k+1} e^{-\frac{ b_1 j^{2k}}{n^{k-1}}} dx \Big)\nonumber\\
        & \leq & C \Big( 1 + n^{(2k+1)(k-1)/2k} \Big) .
    \end{eqnarray*}
    A similar argument can be made for $I_2$. By \eqref{cl3-01} and \eqref{cl2-1}, we have
    \begin{eqnarray}\label{cl3-02}
        \ep{S_n^{2k}} \leq C n^{2k-1}.
    \end{eqnarray}
    This completes the proof of \eqref{cl3-1}.
  A similar argument holds for \eqref{cl3-2}.
  This completes the proof of \cref{cl-3}.
\end{proof}

\begin{lemma}
  For $0 < \beta \leq 1$, there exists a constant
  $b_3>\beta$ such that
  \begin{eqnarray}\label{cl4-1}
      \E e^{b_3 \xi^2 /2} \leq C.
  \end{eqnarray}
  \label{cl-4}
\end{lemma}
\begin{proof}
  When $0 < \beta < 1$, we choose $b_3$ such that
  $\beta <  b_3 < b$;  then,
  \begin{eqnarray*}
    \E e^{b_3 \xi^2/2} &=&
    \frac{1}{\sqrt{2\pi b_3}}
    \int_{-\infty}^{\infty} \E e^{t \xi - t^2/(2b_3)}
    dt\\
    & \le & \frac{1}{ \sqrt{2 \pi b_3}}  \int_{-\infty}^{\infty}
    e^{-\frac{t^2}{2}(\frac{1}{b_3} -
    \frac{1}{b})} dt\\
    & \le & C.
  \end{eqnarray*}
  When $\beta = 1$, we choose $b_3$ such that
  $1 <  b_3 < b_2$. Then, % such that
  \begin{eqnarray*}
    \E e^{b_3 \xi^2/2} &=&
    \frac{1}{\sqrt{2\pi b_3}}
    \int_{-\infty}^{\infty} \E e^{t \xi - t^2/2b_3}
    dt\\
    & \le &  \frac{1}{\sqrt{2\pi b_3}}
    \int_{|t| \leq b_0 } \expb{
    \frac{t^2}{2 } - b_1 t^4 -
  \frac{t^2}{2b_3} } dt \\
&& + \frac{1}{\sqrt{2\pi b_3}}
  \int_{|t| > b_0} \expb{ \frac{t^2}{2b_2} -
  \frac{t^2}{ 2b_3}  } dt \\
  & \le &  C.
  \end{eqnarray*}
  This proves \eqref{cl4-1}.
\end{proof}

\def\s{\bar{X}}
Let $\s_i = \frac{1}{n} (S_n - X_i)$.
%In \cref{cl-5}, let $r \geq 0$ be a finite number and  $C_r$  a constant  depending on $\beta, b, b_2$ and $r$.
\begin{lemma}
  %We have
  For $0 < \beta \leq 1$, and for   $r \geq 1$, we have
  \begin{eqnarray}\label{cl5-1}
    \conep{|X_i|^r}{\s_i} \le C
    e^{\beta \s_i^2}.
  \end{eqnarray}
  %where
  \label{cl-5}
\end{lemma}

\begin{proof}
  Let $\xi$ be a random  variable with the  probability
  measure $\rho$ independent of $\s_i$. Then,
  \begin{eqnarray*}
    {  \conep{|X_i|^r}{\s_i}  }
  &=& \frac{ \conep{|\xi|^r e^{
\frac{ \beta \xi^2}{2n} + \beta \s_i \xi }
}{\bar{X}_i} } {
\conep{ e^{ \frac{ \beta \xi^2}{2n} + \beta \s_i
\xi}}{\bar{X}_i}  },
  \end{eqnarray*}
 % where $\xi$ is independent of $\s_i$,
  and
  \begin{eqnarray}
    {  \conep{  e^{ \frac{ \beta \xi^2}{2n} + \beta
          \s_i
    \xi}}{\s_i} }  \nonumber
    & \ge & \conep{ e^{\beta \s_i \xi}}{\s_i}\\
    & \ge & e^{-\frac{ \beta \s_i^2}{2}}
    \nonumber
    \ep{ e^{- \beta \xi^2/2}}\\
    & \geq & e^{-\frac{\beta \s_i^2}{2}}
    e^{-\beta \ep{\xi^2}/2}\nonumber\\
    & \ge & e^{-\beta/2} e^{-\beta \s_i^2/2}.
    \label{cl-5-1}
  \end{eqnarray}
  By \cref{cl-4}, given  % we have with
  $t = b_3 y$, where $b_3$ depends on $\beta, b$ and $b_2$, we have
  \begin{eqnarray*}
    {\P(|\xi| \ge y)  }
    & \le &   e^{-t y } \ep{ e^{t |\xi|}}\\
    & \le & e^{- ty } \ep{ e^{
  \frac{ b_3\xi^2}{2} + \frac{ t^2}{2b_3}}}\\
  & \le & C e^{-b_3y^2/2}.
  \end{eqnarray*}
  %where we choose $t=b_3y $ in the last inequality.

  Therefore,
  \begin{eqnarray*}
    {  \ep{ |\xi|^r e^{
    \frac{\beta\xi^2}{2n} + \frac{\beta \xi^2}{2}
    } }  }\nonumber
    & \le  & \int_0^{\infty} (ry^{r-1}+ 2\beta
    y^{r+1}) e^{\beta y^2 (1+1/n)/2} \P(|\xi| \ge
    y) dy\\
    & \le & C  \int_0^{\infty} (ry^{r-1}+ 2\beta
    y^{r+1}) e^{\beta y^2 (1+1/n)/2 - b_3y^2 /2} d
    y \nonumber\\
    & \le & C ,\label{l51-0}
  \end{eqnarray*}
 and by the Cauchy inequality,
  \begin{eqnarray}
    {  \conep{   |\xi|^r e^{ \frac{ \beta \xi^2}{2n}
          + \beta \s_i
    \xi}}{\s_i} }   \nonumber
    & \le & e^{ \beta \s_i^2/2} \ep{|\xi|^r
    e^{\frac{ \beta \xi^2}{2n} +
  \frac{ \beta \xi^2}{2} } }\nonumber\\
  & \le & C e^{\beta \s_i^2/2}. \label{cl-5-2}
  \end{eqnarray}
  This completes the proof of \eqref{cl5-1}.
\end{proof}
%Let $\theta \geq 0$ be a constant and $C_{r,\theta}$ be a constant depending on $\beta, b, b_2$<++>
\begin{lemma}
If $0 < \beta < 1$ and \eqref{c1-00} is satisfied,
  then
  for  $ r> 0 $ and $ \theta > 0  $, we have
  \begin{eqnarray}
    \ep{ |\bar{X}_i|^r e^{ \theta \bar{X}_i^2} }
    \leq C n^{-r/2}.\label{cl-6-1}
  \end{eqnarray}
  If $\beta = 1$ and \eqref{c1-01} is satisfied,
  then for   $r \geq 0$ and $\theta>0$, we have
  \begin{eqnarray}
    \ep{ |\bar{X}_i|^r e^{ \theta \bar{X}_i^2} }
    \leq C  n^{- \frac{r}{2k}}.\label{cl-6-2}
  \end{eqnarray}
  \label{cl-6}
\end{lemma}

\begin{proof}
    Without loss of generality, assume $i=1$. Observe that
\begin{eqnarray*}
  \ep{ |\bar{X}_1|^r e^{ \theta \bar{X}_1^2} }
  &=& \frac{1}{ n^{r} Z_n} \E| \xi_2 + \cdots +
  \xi_n|^r e^{ \frac{\beta}{2n} ( \xi_1 + \cdots
    + \xi_n)^2  +
    \frac{ \theta}{n^2}
  ( \xi_2 + \cdots + \xi_n)^2   }\\
  & \le &  \frac{1}{n^{r} Z_n} \E e^{
  \frac{\beta}{2}(1+1/n) \xi_1^2 } \E| \xi_2 +
  \cdots + \xi_n |^r e^{ (\frac{\beta}{2n} +
  \frac{\theta + \beta}{n^2})( \xi_2 + \cdots +
\xi_n)^2 }.
\end{eqnarray*}

When $0 < \beta< 1$ and \eqref{c1-00} is satisfied, by \eqref{cl2-0}, we have $Z_n \geq 1$. Also,  similar to \eqref{cwmeq3},
%we have $Z_n\geq 1$ and
\begin{eqnarray*}
  \P( |\xi_2 + \cdots + \xi_n| > y) \leq 2
  e^{- {b y^2 \over 2(n-1)}}.
\end{eqnarray*}
Thus, for $r \geq 2$,
\begin{eqnarray*}
  \lefteqn{  \E| \xi_2 + \cdots +
  \xi_n|^r e^{ (\frac{\beta}{2n} +
    \frac{ \theta+\beta}{n^2})
  ( \xi_2 + \cdots + \xi_n)^2   }  }\\
  &\leq & C \int_0^\infty (r y^{r-1} + (\beta
  n^{-1} + 2(\theta+\beta) n^{-2}) y^{r+1} ) e^{ (\frac{\beta}{2n} +
  \frac{ \theta+\beta}{n^2})y^2 - \frac{b}{2(n-1)}y^2  } dy \\
  & \le & C n^{r/2}.
\end{eqnarray*}
This proves \eqref{cl-6-1}.
Similarly, following the proof of
\eqref{cl3-1}, \eqref{cl-6-2} holds for $r\geq 2$.
When $r = 0$, similar to \cref{cl-2}, we have
\begin{eqnarray*}
    \ep{ e^{\theta \s_i^2} } \leq C.
\end{eqnarray*}
By the Cauchy inequality, \eqref{cl-6-1} and \eqref{cl-6-2} hold for $0 < r < 2.$
This completes the proof of \cref{cl-6}.
\end{proof}

\begin{lemma}
  For each $1\leq i < j  \leq n$, we have
  \begin{eqnarray}\label{cl7-1}
    \lefteqn{| \ep{ (X_i^2 - 1 )(X_j ^2 -1) } | }\nonumber \\ & \leq &
    \begin{cases}
        C n^{-1}, &0< \beta < 1, \text{ under } \eqref{c1-00},\\
        C n^{-1/k}, & \beta = 1, \text{ under } \eqref{c1-01}.
    \end{cases}
  \end{eqnarray}
  \label{cl-7}
\end{lemma}
\begin{proof}
    We consider $i=1,j=2$ only. Note that
  \begin{eqnarray*}
    \ep{(X_1^2 - 1)(X_2^2-1)} =
    \frac{1}{Z_n} \E (\xi_1^2 -1)(\xi_2^2 - 1)
    \expb{ \frac{\beta}{2n} ( \xi_1 + \cdots +
    \xi_n)^2 }.
  \end{eqnarray*}
  Set $m_{12}  = \xi_3 + \cdots + \xi_n$.
  We first calculate the conditional expectation
  given $\xi_3,\cdots,\xi_n$. In fact,
  for any $s$, we have
  \begin{eqnarray*}
    \lefteqn{  \E (\xi_1^2 - 1)(\xi_2^2-1)
    e^{\frac{\beta}{2n}(\xi_1+\xi_2)^2 +
  \frac{\beta}{n}(\xi_1 + \xi_2)s}  }\\
  &=& \frac{\sqrt{\beta}}{\sqrt{2\pi}}
  \int_{-\infty}^\infty \E (\xi_1^2 -
  1)(\xi_2^2-1) \expb{ \frac{\beta
  t}{\sqrt{n}}(\xi_1 + \xi_2) +
\frac{\beta s}{n}(\xi_1 + \xi_2) -
\frac{\beta t^2 }{2} } dt\\
&=& \frac{\sqrt{\beta}}{\sqrt{2\pi}}
\int_{-\infty}^\infty \left( \ep{ (\xi_1^2 - 1)
e^{(\frac{\beta t }{\sqrt{n} } +
\frac{\beta s}{n})\xi_1 } } \right)^2
e^{-\beta t^2/2} dt.
  \end{eqnarray*}

  Observe that
  \begin{eqnarray}\label{cl7-01}
    \lefteqn{  \left| \E (\xi_1^2 - 1)
    \expb{\frac{\beta t}{\sqrt{n}} \xi_1 +
  \frac{\beta s}{n } \xi_1  } \right|  }\nonumber\\
  & \le & \left( \frac{\beta t}{\sqrt{n}} +
  \frac{\beta s}{n} \right) \E (|\xi_1|^3 +
  |\xi_1|)  \expb{\frac{\beta t}{\sqrt{n}}| \xi_1| +
  \frac{\beta s}{n } | \xi_1|  }\nonumber\\
  & \le &   \left( \frac{\beta t}{\sqrt{n}} +
  \frac{\beta s}{n} \right) e^{\beta s^2 /(2n^2) +
  \beta t^2 /(2\sqrt{n})} \E(|\xi_1|^3 + |\xi_1| )
  e^{ \frac{\beta \xi_1^2}{2}(1 +
  \frac{1}{\sqrt{n}}) } \nonumber\\
  & \le &  C \left( \frac{\beta t}{\sqrt{n}} +
  \frac{\beta s}{n} \right) e^{\beta s^2 /(2n^2) +
  \beta t^2 /(2\sqrt{n}) }.
  \end{eqnarray}
  Therefore,
  \begin{eqnarray*}
    \lefteqn{  \Big| \E (\xi_1^2 - 1)(\xi_2^2-1)
    e^{\frac{\beta}{2n}(\xi_1+\xi_2)^2 +
  \frac{\beta}{n}(\xi_1 + \xi_2)s}   \Big| }\\
  & \le &  C \int_{-\infty}^{\infty} \left(
  \frac{t^2}{n} + \frac{s^2}{n^2}   \right)
  \expb{ \frac{\beta t^2 }{\sqrt{n}} +
  \frac{ \beta s^2}{n^2} - \frac{\beta t^2}{2} }
  dt\\
  & \le &  C \left( \frac{1}{n} +
  \frac{s^2}{n^2} \right ) e^{\beta s^2 /n^2}.
  \end{eqnarray*}
 Hence,
 \begin{eqnarray*}
     \left| \ep{(X_1^2-1)(X_2^2-1)}\right| \leq C \E
   \left( \frac{1}{n} +
   \frac{m_{12}^2}{n^2} \right ) e^{\beta
   m_{12}^2 /n^2 + {\beta m_{12}^2}/{(2n)}}.
 \end{eqnarray*}
 Similar to the proofs of \cref{cl-3,cl-6},  % and \cref{cl-6},
 for $0 < \beta < 1$,
 \begin{eqnarray*}
   \E
   \left( \frac{1}{n} +
   \frac{m_{12}^2}{n^2} \right ) e^{\beta
   m_{12}^2 /n^2 + {\beta m_{12}^2}/{(2n)}} \leq C
   n^{-1},
 \end{eqnarray*}
 and for  $\beta = 1$,
 \begin{eqnarray*}
    \E
    \left( \frac{1}{n} +
   \frac{m_{12}^2}{n^2} \right ) e^{\beta
   m_{12}^2 /n^2 + {\beta m_{12}^2}/{(2n)}} \leq C
   n^{-1/k}.
 \end{eqnarray*}
 This completes the proof of \eqref{cl7-1}.
\end{proof}

For $1 \leq i \leq n$, let $\X =
\sigma(X_1,\dots,X_n)$, and
\begin{eqnarray*}
    Q_i = \conep{ (X_i - X_i')|X_i - X_i'| }{\X}.
\end{eqnarray*}
As defined at the beginning of this subsection, given $\left\{ X_j, j\neq i \right\},$ $X_i'$ and $X_i$ are conditionally independent and have the same distribution.
%and then we have the following result.
\begin{lemma}
    For $0 < \beta \leq 1$, we have
  \begin{eqnarray}
      \ep{Q_i^2} & \leq &  C,\label{cl8-0}\\
    |\ep{Q_i Q_j }|& \leq &
    C n^{-1}. % , \quad 0 < \beta \leq 1.
    \label{cl8-1}
  \end{eqnarray}
  \label{cl-8}
\end{lemma}

\begin{proof}
    By \cref{cl-5,cl-6}, % and \cref{cl-6},
  \begin{eqnarray*}
      \ep{Q_i^2} \leq \ep{X_i - X_i'}^2 \leq 4
      \ep{X_i^2} \leq C.
  \end{eqnarray*}

  To prove
  \eqref{cl8-1},
let
\begin{eqnarray*}
  u(s,t) = (s-t)|s-t|.
\end{eqnarray*}
Let $\xi, \xi_1,\cdots,\xi_n$ be i.i.d. random
variables with probability measure $\rho$, which are independent of $(X_1,...,X_n)$.
We have
\begin{eqnarray*}
  Q_i = \frac{ \conep{ u(X_i, \xi) \expb{
  \frac{\beta \xi^2}{2n} + \beta \bar{X}_i \xi }
  }{\X}  }{ \conep{ \expb{ \frac{\beta \xi^2}{2n} +
  \beta \bar{X}_i \xi } }{\X} }.
\end{eqnarray*}

Without loss of generality,  consider $i=1,j=2$.
Define $\bar{X}_{12} =
\frac{1}{n}(S_n - X_1 - X_2)$, and
\begin{eqnarray*}
    Q_1' &=& \frac{ \E \Big( u(X_1,\xi) \exp(\beta \s_{12} \xi) \Big\vert \X \Big)  }{  \E \Big(  \exp(\beta \s_{12} \xi) \Big\vert \X \Big)   },\nonumber\\
    Q_2' &=& \frac{ \E \Big( u(X_2,\xi) \exp(\beta \s_{12} \xi) \Big\vert \X \Big)  }{  \E \Big(  \exp(\beta \s_{12} \xi) \Big\vert \X \Big)   }.
\end{eqnarray*}
Again, let
$m_{12} = (\xi_3 + \cdots + \xi_n)$.
We have
\begin{eqnarray*}
  \lefteqn{\ep{Q_1'Q_2'} } \\ &=&  \frac{1}{Z_n} \E
  \tilde{u}(\xi_1,m_{12}) \tilde{u}(\xi_2,
  m_{12}) \expb{ \frac{\beta }{2n}(\xi_1 +
    \xi_2)^2 + {\beta \over n} (\xi_1+\xi_2) m_{12} +
\frac{ \beta}{2n} m_{12}^2  }
\end{eqnarray*}
where
\begin{eqnarray*}
  \tilde{u}( x , y) =
  \frac{ \ep{ u( x, \xi) e^{
     {\beta \over n}y  \xi }
    }  }{ \ep{e^{
      {\beta \over n}
   y \xi }
} } .
\end{eqnarray*}
As $u(x,y)$ is anti-symmetric, we have
\begin{eqnarray*}
    \conep{ \tilde{u}(\xi_1,m_{12}) \tilde{u} (\xi_2, m_{12}) e^{\frac{\beta}{n}(\xi_1 + \xi_2) m_{12}}  }{m_{12}} = 0.
\end{eqnarray*}
Moreover,
\begin{eqnarray*}
    \lefteqn{  \ep{ |u(x,\xi)| e^{\frac{\beta y \xi }{n}} }  }\\
    & \leq &  C \Big(  x^2 \E e^{\beta y \xi /n  } + \ep{ \xi^2 e^{\beta y \xi /2}} \Big)
    \nonumber\\
    & \leq &  C e^{C y^2 /n^2 } \big( 1 +  x^2 + y^2/n^2  \big) .
\end{eqnarray*}
Similar to  \eqref{cl-5-1}, $\E e^{\beta y \xi /n} \geq C e^{-Cy^2/n^2}$, and thus,
\begin{eqnarray*}
    | \tilde{u}(x,y)| \leq C e^{C y^2/n^2} (1 + x^2 + y^2/n^2).
\end{eqnarray*}
Therefore, similar to \cref{cl-5,cl-6}, % and \cref{cl-6},
\begin{eqnarray}\label{cl8-4}
    \lefteqn{ | \E (Q_1' Q_2')| } \nonumber\\
    & \leq &  \frac{\beta }{n Z_n } \E |\tilde{u}(\xi_1, m_{12}) \tilde{u}(\xi_2, m_{12})  | (\xi_1 + \xi_2)^2 e^{\frac{\beta}{2n}(\xi_1+\dots+\xi_n)^2}\nonumber\nonumber\\
    & \leq &  \frac{ C }{n Z_n} \E \Big( 1 + \xi_1^4 + \xi_2^4 + \frac{m_{12}^4 }{n^4}  \Big) (\xi_1^2 + \xi_2^2)  e^{\frac{\beta}{2n}(\xi_1+\dots+\xi_n)^2}\nonumber\nonumber\\
    & \leq &  \frac{C}{n } \E (1 + \s_{12}^4) (1 + X_1^6 + X_2^6 ) e^{C \s_{12}^2} \nonumber\\
    & \leq &  \frac{C}{n}.
\end{eqnarray}

Next, we estimate
$\ep{(Q_1-Q_1')^2}$.
Note that
\begin{eqnarray*}
  \lefteqn{  |Q_1 - Q_1'|  }\\
  & \le &  \frac{ \Big| \E \Big( u(X_1,\xi) e^{\beta \s_{12} \xi} ( e^{ \frac{\beta \xi^2}{2n} + \frac{\beta X_2}{n}} - 1 ) \Big\vert \X \Big)  \Big|  }{ \E \expb{ \frac{\beta \xi^2}{2n } + \beta \s_1 \xi } }\\
  && +  \frac{ \conep{ |u(X_1,\xi)| e^{ \beta \s_{12} \xi}   }{\X} \conep{  e^{\beta \s_{12} \xi} | e^{ \frac{\beta \xi^2}{2n} + {\beta \s_{12} \xi} }-1 |  }{\X}   }{ \conep{ e^{\frac{\beta \xi^2}{2n} + \beta \s_{12} \xi } }{\X} \conep{ e^{\beta \s_{12} \xi}}{\X}   }.
\end{eqnarray*}
Note also that $|u(s,t) |\leq (s-t)^2$.  Similar to \cref{cl-5,cl-6}, % and
%\cref{cl-6},
we have
\begin{eqnarray}\label{cl8-5}
  \ep{(Q_1 - Q_1')^2} \leq C n^{-2}.
\end{eqnarray}
Observe that
\begin{eqnarray}\label{cl8-2}
  | \ep{Q_1 Q_2}|& \leq &  |\ep{Q_1'Q_2'}  |
  + | \ep{Q_1 (Q_2 - Q_2')} |\nonumber\\ &&  + | \ep{Q_2 (Q_1 -
  Q_1')} | +  |\E (Q_1  - Q_1')(Q_2 - Q_2') |,
\end{eqnarray}
Then, by  the Cauchy inequality and substituting
\eqref{cl8-0}, \eqref{cl8-4} and \eqref{cl8-5} into
\eqref{cl8-2}, we get the desired
result.
\end{proof}
We are now ready to prove  Propositions \ref{l-c1}--\ref{l-c3}.

\begin{proof}[Proof of \cref{l-c1}]
    By the definition of $S_n$ and $S_n'$, we have
%Then
\begin{eqnarray*}
  \conep{S_n-S_n'}{\X} & = & {1\over n}\sum_{i=1}^n \conep{X_i - X_i'}{\X}\\
  & = & \s - {1\over n}\sum_{i=1}^n \conep{X_i'}{\X}\\
  & = & \s - {1\over n} \sum_{i=1}^n { \int_{-\infty}^{\infty}x
  e^{{\beta x^2\over 2n}+ \beta \s_i x} d\rho(x)
\over \int_{-\infty}^{\infty}e^{ {\beta x^2 \over 2n} + \beta \s_i x } d\rho(x)   }.
\end{eqnarray*}
Observe that for $0 < \beta \leq 1$,
\begin{eqnarray}\label{lc1-h0}
  { \int_{-\infty}^{\infty}x e^{{\beta x^2\over 2n}+ \beta \s_i x}
  d\rho(x) \over \int_{-\infty}^{\infty}e^{ {\beta x^2 \over 2n} +
\beta \s_i x } d\rho(x)   } = h(\s_i)
+ r_{1i},
\end{eqnarray}
where
\begin{eqnarray*}
    h(s)& =& { \int_{-\infty}^{\infty}x e^{\beta s x} d\rho(x) \over \int_{-\infty}^{\infty}e^{  \beta s x } d\rho(x)   } \text{ and }\\
r_{1i} &= &{ \int_{-\infty}^{\infty}x e^{{\beta x^2\over 2n}+ \beta
\s_i x} d\rho(x) \over \int_{-\infty}^{\infty}e^{ {\beta x^2 \over
2n} + \beta \s_i x } d\rho(x)   } - { \int_{-\infty}^{\infty}x
e^{\beta \s_i x} d\rho(x) \over \int_{-\infty}^{\infty}e^{  \beta \s_i x } d\rho(x)   }.
\end{eqnarray*}

We first give the  bound of  $\E|r_{1i}|$.
Note that by \eqref{cl-5-1} and \eqref{cl-5-2},
\begin{eqnarray}\label{lc1-r1}
  \E|r_{1i}| & \leq & \E \Bigg| {\int_{-\infty}^{\infty}x \big(
  e^{\beta x^2 \over 2n}-1  \big)e^{\beta \s_i x}
d\rho(x) \over {\int_{-\infty}^{\infty}e^{\beta x^2 \over 2n}
e^{\beta \s_i x} d\rho(x)  }  } \Bigg| \nonumber\\ && + \E
\Bigg| {\int_{-\infty}^{\infty} \big( e^{\beta x^2 \over 2n}-1
\big)e^{\beta \s_i x} d\rho(x) \int_{-\infty}^{\infty}x  e^{\beta
x^2 \over 2n}  e^{\beta \s_i x} d\rho(x)\over
{\int_{-\infty}^{\infty}e^{\beta x^2 \over 2n} e^{\beta \s_i x}
d\rho(x)  \int_{-\infty}^{\infty} e^{\beta \s_i x} d\rho(x)}  } \Bigg|\nonumber\\
& \le & \frac{C}{n} \E \left|
\frac{ \int_{-\infty}^{\infty}|x|^3  \expb{ \frac{\beta x^2}{2n} +
\beta \s_i x } d\rho(x) }{ \int_{-\infty}^{\infty} \exp^{ \beta \s_i
x} d \rho(x) }
\right|
\nonumber\\
&& + \frac{C}{n} \E \left| \frac{ \int_{-\infty}^{\infty}|x|^2 e^{
\frac{\beta x^2}{2n} + \beta \s_i x }
d\rho(x) \int_{-\infty}^{\infty}|x| e^{ \frac{\beta x^2}{2n} +
\beta \s_i x  } d\rho(x)  }{ (\int_{-\infty}^{\infty} e^{ \beta \s_i
x} d \rho(x))^2  }  \right|\nonumber\\
& \leq & C n^{-1}\E  e^{C \s_i^2} \nonumber\\
  & \leq & C n^{-1}.
\end{eqnarray}
For $h(\s_i)$, we consider two cases.
\begin{description}
    \item [Case 1. ] {\bf $\beta=1$.} As $\rho$ is of type $k$, by the Taylor expansion,
\begin{eqnarray}\label{lc1-h1}
  h(\s_i) & = &  \s_i + {h^{(2k-1)}(0) \over
  (2k-1)!} \s_i^{2k-1}+ {1\over
  (2k-1)!}\int_0^{\s_i} h^{(2k)}(t) (\s_i-t)^{2k-1} dt\nonumber\\
  & = &  \s - {1\over n} X_i  +
  {h^{(2k-1)}(0) \over (2k-1)!} \s^{2k-1}
  + {h^{(2k-1)}(0) \over
  (2k-1)!}(\s_i^{2k-1}- \s^{2k-1})\nonumber\\
  && +{1\over (2k-1)!}\int_0^{\s_i} h^{(2k)}(t)
  (\s_i-t)^{(2k-1)} dt.
\end{eqnarray}

Hence,
\begin{eqnarray}
  \conep{S_n-S_n'}{\X} = { {h^{2k-1}(0)\over
  (2k-1)!} \s^{2k-1} } + R_1,\label{cwma11}
\end{eqnarray}
where
\begin{eqnarray*}
    R_1 = - \frac{1}{n}\sum_{i=1}^n \Big( h(\s_i) - \s - \frac{h^{(2k-1)}(0)}{(2k-1)!}\s^{2k-1} \Big) - \frac{1}{n}\sum_{i=1}^n r_{1i},
\end{eqnarray*}
and $r_{1i}$ is given in \eqref{lc1-h0} with $\beta=1$.

Observe that by \eqref{lc1-h1},
\begin{eqnarray}\label{lc1-h2}
    \lefteqn{  h(\s_i) - \s - \frac{h^{(2k-1)}(0)}{(2k-1)!}\s^{2k-1}   }\nonumber\\
        &=& - {1\over n} X_i
  + {h^{(2k-1)}(0) \over
  (2k-1)!}(\s_i^{2k-1}- \s^{2k-1})\nonumber \\ &&
  +{1\over (2k-1)!}\int_0^{\s_i} h^{(2k)}(t)
  (\s_i-t)^{(2k-1)} dt.
\end{eqnarray}

For the first  term of \eqref{lc1-h2}, it follows from  \cref{cl-5,cl-6} that % and
%\cref{cl-6} that
\begin{eqnarray}\label{lc1-h3}
  {1 \over n} \E|X_i| \leq C n^{-1}.
\end{eqnarray}

For the second term, by \cref{cl-5,cl-6}  again,
\begin{eqnarray}\label{lc1-h4}
  \lefteqn{ {h^{(2k-1)}(0) \over
(2k-1)!}\E|\s_i^{2k-1}- \s^{2k-1}|  }\nonumber\\
& \leq & Cn^{-1} \E\Big( |X_i|(|\s_i|^{2k-2} + (|X_i|/n)^{2k-2})\Big) \nonumber\\
  & \leq & Cn^{-1} \E(1 +
  |\s_i|^{2k-1}) e^{C |\s_i|^2}\nonumber\\
  % & \leq & C n^{-1} \ep{ e^{2C |\s_i |^2} }\nonumber\\
  & \leq & Cn^{-1}.
\end{eqnarray}

To bound the last term, we first consider $h^{(2k)}(s)$. Recalling that
\begin{eqnarray*}
  h(t) = { \int_{-\infty}^{\infty}x e^{t x} d\rho(x) \over \int_{-\infty}^{\infty}e^{t x} d\rho(x)}
\end{eqnarray*}
and observing that
$$
\int_{-\infty}^\infty e^{tx} d \rho(x) \geq 1
$$
and
$$
|{ d^{j} \over dt^{j} } \int_{-\infty}^\infty e^{tx} d \rho(x)|
= |\int_{-\infty}^{\infty} x^j e^{tx} d\rho(x)|
\leq \int_{-\infty}^\infty (1+ |x|^{2k+1}) e^{tx} d \rho(x)
$$
for $ j=0, 1, \cdots, 2k+1$, we have
\begin{eqnarray*}
  |h^{(2k)}(t) |
   & \leq & C   \int_{-\infty}^{\infty} (1+ |x|^{2k+1})  e^{t x} d\rho(x)\\
  & \leq & C e^{t^2/2} ,
\end{eqnarray*}
Thus, by \eqref{cl-6-2},
\begin{align}\label{lc1-h5}
  \ML  {1\over (2k-1)!}\E\Big|\int_0^{\s_i}
  h^{(2k)}(t) (\s_i-t)^{(2k-1)} dt\Big| \nonumber\\
  & \leq  C \ep{ \s_i^{2k} e^{\s_i^2/2} } \leq
    C n^{-1}.
\end{align}
By \eqref{lc1-h3}, \eqref{lc1-h4} and \eqref{lc1-h5}, \eqref{lc1-h2} can be bounded by
\begin{eqnarray}
    \label{lc1-h6}
    \E \Big|  h(\s_i) - \s - \frac{h^{(2k-1)}(0)}{(2k-1)!}\s^{2k-1}  \Big| \leq C n^{-1}.
\end{eqnarray}

Together with  \eqref{lc1-h0} and \eqref{lc1-r1}, we have
\begin{eqnarray*}
    \E|R_1| \leq C n^{-1}.
\end{eqnarray*}
\item [Case 2. ]
For $\beta \in (0,1)$, we have
\begin{eqnarray*}
    h(\s_i) &=& \beta \s_i + \int_0^{\s_i} h''(t) (\s_i - t)dt\\
    &=& \beta \s - \frac{\beta }{n} X_i + \int_0^{\s_i} h''(t) (\s_i - t)dt.
\end{eqnarray*}
Hence, %We also have
\begin{eqnarray*}
    \conep{S_n - S_n'}{\X} &=& (1-\beta)\s + R_2,
\end{eqnarray*}
where
\begin{eqnarray*}
    R_2 = - \frac{1}{n}\sum_{i=1}^n \Big( - \frac{\beta}{n} X_i + \int_0^{\s_i} h''(t) (\s_i - t) dt \Big) - \frac{1}{n}\sum_{i=1}^n r_{1i}.
\end{eqnarray*}
Similar to \eqref{lc1-h6}, we have
\begin{eqnarray*}
    \E \Big|  - \frac{\beta}{n} X_i + \int_0^{\s_i} h''(t) (\s_i - t) dt \Big| \leq C n^{-1}.
\end{eqnarray*}
Together with \eqref{lc1-r1}, we have
\begin{eqnarray*}
    \E|R_2| \leq C n^{-1}.
\end{eqnarray*}
\end{description}
%Similar to \eqref{cwma11}, we have \eqref{lc1-1} holds.
This completes the proof.
\end{proof}

\begin{proof}[Proof of \cref{l-c2}]
    Observe that
\begin{eqnarray*}
  \conep{(S_n-S_n')^2}{\X} & = & {1\over n}\sum_{i=1}^n \conep{X_i^2 - 2 X_i X_i' + (X_i')^2}{\X}\\
  & = & {1\over n } \sum_{i=1}^n\Bigg( X_i^2 -
  {2X_i \int_{-\infty}^{\infty}x e^{{\beta x^2\over 2n}+\beta \s_i
x} d\rho(x) \over \int_{-\infty}^{\infty}e^{{\beta x^2\over
2n}+\beta \s_i x} d\rho(x) }  \nonumber \\ && \qquad \qquad +  { \int_{-\infty}^{\infty}x^2
e^{{\beta x^2\over 2n}+\beta \s_i x} d\rho(x)
\over \int_{-\infty}^{\infty}e^{{\beta x^2\over 2n}+\beta \s_i x} d\rho(x) } \Bigg)\\
& :=& 2 + R_3 + R_4 + R_5,
\end{eqnarray*}
where
\begin{eqnarray*}
    R_3& = & {1\over n}\sum_{i=1}^n (X_i^2 -1), \\
  R_4 & = & -{1\over n} \sum_{i=1}^n {2X_i  \int_{-\infty}^{\infty}x
  e^{{\beta x^2\over 2n}+\beta \s_i x} d\rho(x)
\over \int_{-\infty}^{\infty}e^{{\beta x^2\over 2n}+\beta \s_i x} d\rho(x) }, \\
  R_5 & = & {1\over n}\sum_{i=1}^n { \int_{-\infty}^{\infty}x^2
  e^{{\beta x^2\over 2n}+\beta \s_i x} d\rho(x)
\over \int_{-\infty}^{\infty}e^{{\beta x^2\over 2n}+\beta \s_i x} d\rho(x) }-1.
\end{eqnarray*}
By the Taylor expansion, and similar to the proof of $\E|R_1|$ and $\E|R_2|$, we have
\begin{eqnarray*}
   \E|R_4| + \E|R_5| \leq
   \begin{cases}
   C n^{-1/2}, &  0 < \beta < 1,\\
   C n^{ -\frac{1}{2k}}, & \beta = 1,
 \end{cases}
 \end{eqnarray*}
As for $\E|R_3|,$ we have %that
\begin{eqnarray*}
    \E(R_3^2) = {1\over n^2}\sum_{i=1}^n \E (X_i ^2 - 1)^2  + {1\over n^2}\sum_{i\neq j}
  \E(X_i^2 -1)(X_j^2 -1).
\end{eqnarray*}
By \cref{cl-7}, we have
\begin{eqnarray*}
  \E(X_i^4) \leq C,\quad
  |\E(X_i^2-1)(X_j^2 -1)| \leq
  \begin{cases} C n^{-1}  , & 0 <  \beta < 1, \\ C
      n^{-1/k},  & \beta = 1 .
\end{cases}
\end{eqnarray*}
Therefore, $$\E|R_3|\leq  \begin{cases} C
    n^{-1/2}, &  0 < \beta < 1, \\ C n^{-{1\over 2k} }, & \beta
= 1. \end{cases}$$
This proves \eqref{lc2-1} and \eqref{lc2-2}.
\end{proof}

\begin{proof}[Proof of \cref{l-c3}]
We have
\begin{eqnarray*}
  \lefteqn{\conep{(S_n-S_n')|S_n-S_n'|}{\X}}\\
  & = & {1\over n }\sum_{i=1}^n \conep{(X_i-X_i')|X_i-X_i'|}{\X}.
\end{eqnarray*}
%It can be shown that $\E|R_6|\leq c_\beta n^{-1}$. For the rest, it is just the sum of functions of $X_i$. Similar to the bound of $\conep{(S_n-S_n')^2}{\X}$, we have that
Then, \eqref{lc3-1} follows from  \cref{cl-8}.
\end{proof}

\subsection{Proof of \cref{thmhm}}
\def\k{\kappa}
The Berry-Esseen bound \eqref{res:hm} follows from \cref{thm-1} and \cref{lemhm} below.
\begin{proposition}
  \label{lemhm}
  Let $W_n$ be as defined in \eqref{def:W} and $\sigma'=\{\sigma_1',...,\sigma_n'\}$, where for each $i$,
  $\sigma_i'$ is an independent copy of $\sigma_i$ given $\{\sigma_j,j\neq i\}$.
  Let $I$ be a random index independent of all
  others and  uniformly distributed over $\{1,...,n\}$, and let
  $W_n'=\sqrt{n} \big( \frac{\beta^2 }{n^2\kappa^2} | S_n'| ^2 -1 \big) $, where $S_n' = \sum_{j=1}^n \sigma_j - \sigma_I + \sigma_I'$. Then, $(W_n,W_n')$ is an exchangeable pair and
there exists a constant $c_\beta$ depending on
$\beta$ only such that
  %\begin{enumerate}[i)]
    \begin{eqnarray}
      \conep{W_n-W_n'}{W_n} = \lambda(W_n-R_n) & \mbox{and} & \E |R_n| \leq
      {c_{\beta} n^{-1/2}},\label{hmc1}
    \end{eqnarray}
    where
      $\lambda={1-\beta \g'(\k) \over n}$;
        \begin{eqnarray}
          \E\Big| B^2 - {1\over 2\lambda}\conep{(W_n-W_n')^2}{W_n}  \Big|\leq
          {c_\beta n^{-1/2}},\label{hmc2}
        \end{eqnarray}
        where $B$ is defined in \eqref{def:B};
       % \item
        and
        \begin{eqnarray}
          \label{hmc3}
          {1\over \lambda}\E\Big| \conep{(W_n-W_n')|W_n-W_n'|}{W_n} \Big|\leq
          {c_\beta   n^{-1/2}}.
        \end{eqnarray}
  %\end{enumerate}
\end{proposition}
\def\sigmai{\sigma^{(i)}}
\begin{proof}
    Let  $S_n = \sum_{i=1}^n \sigma_i$ and $\sigmai = S_n - \sigma_i$. The proof is organized in the following three parts.
  \begin{enumerate}[(i)]
    \item {\bf Proof of \eqref{hmc1}.}
        Let $\sigma=(\sigma_1,\dots,\sigma_n)$.
        As shown in \citet{KirMec2013} (p. 23, equation (12)),
      we have
  \begin{eqnarray*}
    \conep{W_n- W_n'}{\sigma} = {2\over n} W_n + {2\over \sqrt{n}} - {2\beta\over
    n^{1/2}\k^2 } (\beta|S_n|/n) \g(\beta|S_n|/n) + R_1,
  \end{eqnarray*}
  where $\g(x) = \coth(x) - 1/x$ and  $|R_1| \leq C n^{-3/2}$ for some constant $C$ depending on $\beta$.  %It is also shown in \citet{KirMec2013} %Kirkpatrick and Meckes (2013)
 % that $\beta|S_n| / n$
 % converges to $\k$ in probability.
  The Taylor
  expansion yields
  \begin{eqnarray*}
    (\beta |S_n| /n)  \g(\beta |S_n|/n) &=&
    \k\g(\k) + (\g(\k)+\k\g'(\k)) \Big(
    {\beta |S_n| \over n} - \k
    \Big) + R_2
  \end{eqnarray*}
  where $|R_2|\leq C  (\beta |S_n| /n-\k)^2 $ with
  $C$ depending on
  $\beta$.

  Moreover, by  \citet{KirMec2013} (p. 25),
  \begin{eqnarray*}
      {\beta |S_n| \over n} - \k &=& {\k W_n \over 2\sqrt{n}} + R_3,
  \end{eqnarray*}
  where $|R_3|\leq C|W_n|^2/n.$
  Recalling \eqref{fun:g} and combining all of the preceding inequalities, we have
  \begin{eqnarray*}
    \conep{W_n-W_n'}{\sigma} &=& {1-\beta\g'(\k)\over n}
    ( W_n - R_n),
  \end{eqnarray*}
  where $|R_n|\leq CW_n^2/n^{1/2}$.
  It follows from  \citet*{KirMec2013} (p. 24)
  that there exists $\varepsilon_0 >0$ such that
  for all $x \in (0,\varepsilon_0]$,
  \begin{eqnarray*}
    \P\Big(\Big| {\beta |S_n| \over
    n}-\k\Big| > x    \Big)
    \leq e^{-K_\beta n x^2}
  \end{eqnarray*}
  for some constant $K_\beta> 0 .$ Then,
  \begin{eqnarray}\label{hl-1}
      \E|\beta|S_n|/n-\k|^4& \leq & 4
      \int_0^{\varepsilon_0} x^3 e^{-K_\beta n
      x^2} dx + C  \P\Big(\Big|\frac{\beta|S_n|}{n} -\k \Big|
      > \varepsilon_0\Big) \nonumber\\
      & \leq & C n^{-2} + C  e^{-K_\beta n
      \varepsilon_0}\nonumber\\
      & \leq & C
      n^{-2}.
  \end{eqnarray}
  It follows from the definition of $W_n$ that
  \begin{eqnarray*}
    \E|W_n|^2 &=& n \E\Big| {\beta^2 |S_n|^2 \over
    n^2 \k^2}  -1   \Big|^2\\
    & \leq & C n \E\Big|  {\beta |S_n| \over n \k} -1   \Big|^2\\
    & \leq & C,
  \end{eqnarray*}
  where $C$ depends on $\beta$. This proves \eqref{hmc1}.

\item {\bf Proof of \eqref{hmc2}.}
    From \citet{KirMec2013} (pp. 25--27, equations (16) and (18)), we have
   \begin{eqnarray*}
     \lefteqn{\conep{(W_n-W_n')^2 }{\sigma}} \\
     & =& {4\beta^4 \over n^4 \k^4 } \sum_{i=1}^n
     |\sigma^{(i)}|^2 \Big( (1-2\g(b_i)/b_i) - 2\g(b_i) \cos\alpha_i + \cos^2\alpha_i
     \Big)\nonumber\\
     &=& 2\lambda B^2 +  \frac{4\beta^4}{n^4 \kappa^4 } \bigg( % 2 \Big( 1 - \frac{2}{\beta}\Big)  \frac{(n-1)^2 \kappa^2 }{\beta^2 }- \frac{2n^2 \kappa^4}{\beta^4} \Big)
      \sum_{i=1}^n \Big( 1 - \frac{2}{\beta}\Big)  \Big( |\sigma^{(i)}|^2 - \frac{(n-1)^2 \kappa^2}{\beta^2} \Big) \nonumber\\ && \qquad \qquad \qquad  - \frac{2\kappa}{\beta} \sum_{i=1}^n \Big( |\sigma^{(i)}|^2 \cos \alpha_i - \frac{n^2 \kappa^3 }{\beta^3} \Big)  \nonumber\\
 && \qquad\qquad \qquad  + \sum_{i=1}^n \Big( |\sigma^{(i)} |^2 \cos^2 \alpha_i - \Big( 1-\frac{2}{\beta}\Big)  \frac{(n-1)^2 \kappa^2 }{\beta^2} \Big)     \bigg) \nonumber\\
 && + \frac{4\beta^4 }{n^4 \kappa^4}  \sum_{i=1}^n \Big( 2 |\sigma^{(i)}|^2 \Big( \frac{\psi(b_i)}{b_i } - \frac{1}{\beta} \Big) - 2 |\sigma^{(i)}|^2 \cos \alpha_i \Big( \psi(b_i ) - \frac{\kappa}{\beta} \Big)  \Big) ,
   \end{eqnarray*}
   where $b_i = \beta|\sigmai|/n$ and $\alpha_i$ is the angle between $\sigma_i$ and $\sigma^{(i)}$.
   Therefore,
   \begin{eqnarray}
       \lefteqn{\frac{1}{2\lambda} \ep{ \conep{(W_n-W_n')^2 }{ \sigma} } - B^2 }\nonumber\\  &=&  \frac{2\beta^4}{n^3 \kappa^4 (1-\beta \psi'(\kappa))} (R_4 + R_5 + R_6 + R_7)
       , \label{hl-2}
   \end{eqnarray}
   where
   \begin{eqnarray*}
       R_4 &=& \sum_{i=1}^n \Big( 1-\frac{2}{\beta} \Big)  \Big( |\sigma^{(i)}|^2 - \frac{(n-1)^2 \kappa^2 }{\beta^2 } \Big) ,\\
       R_5 &=& \frac{2\kappa}{\beta} \sum_{i=1}^n \Big( |\sigma^{(i)}|^2 \cos \alpha_i - \frac{n^2 \kappa^3 }{\beta^3} \Big), \\
       R_6 &=& \sum_{i=1}^n \Big( |\sigma^{(i)}|^2 \cos^2  \alpha_i - \Big( 1 - \frac{2}{\beta} \Big) \frac{(n-1)^2 \kappa^2 }{\beta^2 } \Big) , \\
       R_7 &=&  \sum_{i=1}^n \Big( 2 |\sigma^{(i)}|^2 \Big( \frac{\psi(b_i)}{b_i } - \frac{1}{\beta} \Big) - 2 |\sigma^{(i)}|^2 \cos \alpha_i \Big( \psi(b_i ) - \frac{\kappa}{\beta} \Big)  \Big)    .
   \end{eqnarray*}
   For $R_4$, note  that $|\sigma^{(i)} - S_n  |\leq 1$; then, by  \eqref{hl-1},
   \begin{eqnarray}
       \label{hl-3}
       \E \Big| \frac{\beta |\sigma^{(i)}|}{n} - \kappa \Big|^4  \leq 8 \E \Big| \frac{\beta |S_n|}{n} - \kappa \Big|^4 + 8/n^4 \leq C n^{-2}.
   \end{eqnarray}
   Thus,
   \begin{eqnarray}
       \label{hl-4}
       \E|R_4 | & \leq &  \sum_{i=1}^n \E \Big| |\sigma^{(i)}|^2 - \frac{(n-1)^2 \kappa^2 }{\beta^2 } \Big|^2 \nonumber\\
       & \leq &  C n^2 \sum_{i=1}^n \E \Big| \frac{\beta^2  |\sigma^{(i)}|^2 }{n^2 } - \kappa^2  \Big|^2 \nonumber\\
       & \leq &  C n^2 \sum_{i=1}^n \E \Big| \frac{\beta |\sigma^{(i)}|}{n} - \kappa \Big|\nonumber\\  & \leq &  C n^{5/2}.
   \end{eqnarray}
   For $R_5$, by  \citet*{KirMec2013} (p. 28), we have
   \begin{eqnarray}
       \label{hl-5}
       \E| R_5| & \leq &   \E \Big| \sum_{i=1}^n \frac{2\kappa}{\beta} \Big( |S_n| \inprod{\sigma_i , S_n } - \frac{n^2 \kappa^3}{\beta^3} \Big)  \Big| + 2\kappa n^2 /\beta \nonumber\\
       & \leq &  \frac{ 2\kappa }{\beta}\E \Big| |S_n|^3 - \frac{n^3 \kappa^3}{\beta^3}  \Big| + 2\kappa n^2 /\beta \nonumber\\
       & \leq &  C n^{5/2}.
   \end{eqnarray}
   For $R_6$, we shall prove shortly that
    \begin{eqnarray}\label{lemh2-1}
        \E \Big( \sum_{i=1}^n \Big( \inprod{\sigma_i, \sigma^{(i)}}^2 - \Big( 1-\frac{2}{\beta} \Big) \frac{(n-1)^2 \kappa^2 }{\beta^2} \Big)  \Big) ^2 \leq C n^{5}.
    \end{eqnarray}
    %%The proof of \eqref{lemh2-1} will be given shortly. %%(on page~\pageref{page-lemh2-1}).
   By  \eqref{lemh2-1} and the Cauchy inequality, we have
   \begin{eqnarray}
       \label{hl-6}
       \E |R_6 | \leq C n^{5/2}.
   \end{eqnarray}
   For $R_7$, as   $\psi(\kappa)/\kappa = 1/\beta$,  and by  the smoothness of $\psi$, we have
   \begin{eqnarray*}
       \Big|  \frac{\psi(b_i)}{b_i} - \frac{\psi(\kappa)}{\kappa}\Big|\leq |b_i - \kappa|,
   \end{eqnarray*}
   and
   \begin{eqnarray*}
       \big| \psi(b_i ) - \psi(\kappa) \big|\leq |b_i - \kappa|.
   \end{eqnarray*}
   Thus, by  \eqref{hl-3},
   \begin{eqnarray}
       \label{hl-7}
       \E|R_7| \leq C n^2 \sum_{i=1}^n\E |b_i - \kappa |\leq C n^{5/2}.
   \end{eqnarray}
   Then,  \eqref{hmc2} follows from \eqref{hl-2}--\eqref{hl-7}.

 \item {\bf Proof of \eqref{hmc3}.}  Similarly,
     we have
     \begin{eqnarray}
         \label{hl-8}
         {  \conep{(W_n-W_n')|W_n - W_n'|}{\sigma}  }
         &=& \frac{4\beta^4}{n^4 \kappa^4} \sum_{i=1}^n M_i,
     \end{eqnarray}
     where
     \begin{eqnarray*}
         M_i = \conep{ \inprod{\sigma_i, \sigma^{(i)}}|\inprod{\sigma_i, \sigma^{(i)}}   | - \inprod{\sigma_i', \sigma^{(i)}}|\inprod{\sigma_i', \sigma^{(i)}}   |   }{\sigma}.
     \end{eqnarray*}
     We shall prove that
    \begin{eqnarray}
        \label{lemh3-1}
        \E \Big( \sum_{i=1}^n M_i  \Big)^2 \leq C n^{5}.
    \end{eqnarray}
    The proof of \eqref{lemh3-1} is given at the end of this subsection. %%on page~\pageref{page-lemh3-1}.

     By  the definition of $\lambda$ and \eqref{lemh3-1}, we have
     \begin{eqnarray*}
         \frac{1}{\lambda} \E \big| \conep{(W_n-W_n')|W_n-W_n'|}{\sigma} \big| \leq C n^{-1/2}.
     \end{eqnarray*}
     This proves \eqref{hmc3}.   Thus, we complete the proof of \cref{lemhm}.
\end{enumerate}%This proves \cref{lemhm}.
\end{proof}
%We must also prove the following two lemmas.
%\begin{lemma}
%    We have
%    \label{lemh-2}
%\end{lemma}

%\begin{lemma}
%    We have
%    \label{lemh-3}
%\end{lemma}
We now give the proofs of \eqref{lemh2-1} and \eqref{lemh3-1}.
\begin{proof}
    [Proof of \eqref{lemh2-1}]
\label{page-lemh2-1}
Set $a =  \big( 1-\frac{2}{\beta}  \big) \frac{(n-1)^2 \kappa^2 }{\beta^2}  $.
    Given the  symmetry, we have
    \begin{eqnarray}\label{lemh2-01}
        { \E \Big( \sum_{i=1}^n \big( \inprod{\sigma_i, \sigma^{(i)}}^2 -a \big) \Big) ^2   }
        %&=& n \E \Big( \inprod{\sigma_1, \sigma^{(1)}}^2 - \Big( 1-\frac{\beta}{2} \Big) \frac{(n-1)^2 \kappa^2 }{\beta^2} \Big) ^2 \nonumber\\
        %&& + n(n-1) \E \Big(  \inprod{\sigma_1, \sigma^{(1)}}^2 - \Big( 1-\frac{\beta}{2} \Big) \frac{(n-1)^2 \kappa^2 }{\beta^2}  \Big) \Big(  \inprod{\sigma_2, \sigma^{(2)}}^2 - \Big( 1-\frac{\beta}{2} \Big) \frac{(n-1)^2 \kappa^2 }{\beta^2}  \Big) \nonumber\\
        &=& H_1 + H_2 ,
    \end{eqnarray}
    where
    \begin{eqnarray*}
        H_1 &=&   n \E \big( \inprod{\sigma_1, \sigma^{(1)}}^2 - a \big) ^2,  \nonumber\\
        H_2 &=&   n(n-1) \E \big(  \inprod{\sigma_1, \sigma^{(1)}}^2 - a \big) \big(  \inprod{\sigma_2, \sigma^{(2)}}^2 - a\big) \nonumber.
    \end{eqnarray*}
    For $H_1$, as $| \sigma^{(1)}| \leq n$, we have
    \begin{eqnarray}
        \label{lemh2-02}
        H_1 \leq C n^5.
    \end{eqnarray}
    For $H_2$, we define  $\sigma^{(1,2)} = S_n - \sigma_1 - \sigma_2$, and for $j = 1,2$, we have
    \begin{eqnarray*}
        | \inprod{\sigma_j, \sigma^{(j)}}^2 - \inprod{\sigma_j, \sigma^{(1,2)}}^2 | \leq C n.
    \end{eqnarray*}
    %`Similarly,
    %`\begin{eqnarray*}
    %`    | \inprod{\sigma_2, \sigma^{(2)}}^2 - \inprod{\sigma_2, \sigma^{(1,2)}}^2 | \leq C n.
    %`\end{eqnarray*}
    Thus,
    \begin{eqnarray}
        \label{lemh2-03}
        H_2 = H_3 + L_1,
    \end{eqnarray}
    where $|L_1| \leq C n^{5}$ and
    \begin{eqnarray*}
        H_3 = n(n-1) \E \big(  \inprod{\sigma_1, \sigma^{(1,2)}}^2 - a  \big) \big(  \inprod{\sigma_2, \sigma^{(1, 2)}}^2 - a \big) \nonumber.
    \end{eqnarray*}
    For $i=1,2$, we define
    \begin{eqnarray*}
        V_i (\sigma^{(1,2)}) = \conep{ \inprod{\sigma_i, \sigma^{(1,2)}}^2 }{\sigma^{(1,2)}},
    \end{eqnarray*}
    and thus,
    \begin{eqnarray}
        \label{lemh2-04}
        \lefteqn{   \E \big(  \inprod{\sigma_1, \sigma^{(1,2)}}^2 - a  \big) \big(  \inprod{\sigma_2, \sigma^{(1, 2)}}^2 - a \big) \nonumber   }\\
        &=& \E \big( \inprod{\sigma_1, \sigma^{(1,2)}}^2  - V_1 (\sigma^{(1,2)})  \big) \big( \inprod{\sigma_2, \sigma^{(1,2)}}^2  - V_2 (\sigma^{(1,2)}) \big) \nonumber\\
        &&\quad  + \; \E \big( V_1 (\sigma^{(1,2)}) - a \big) \big( V_2 (\sigma^{(1,2)}) - a \big).
    \end{eqnarray}
    Note that the conditional probability density function of $(\sigma_1, \sigma_2)$ given $\sigma^{(1,2)}$  is
    \begin{eqnarray}\label{lemh2-05}
        p_{12}(x,y) = \frac{1}{Z_n^{(1,2)}} \exp \Big( \frac{ \beta }{ 2n } \inprod{x,y}^2 + \frac{\beta }{n} \inprod{x+y, \sigma^{(1,2)}}  \Big) ,
    \end{eqnarray}
    where $x,y \in \ss^2$ and
    \begin{eqnarray*}
        Z_n^{(1,2)} = \int_{\ss^2 } \int_{\ss^2} \exp \Big( \frac{ \beta }{ 2n } \inprod{x,y}^2 + \frac{\beta }{n} \inprod{x+y, \sigma^{(1,2)}}  \Big)  dP_n(x) dP_n(y).
    \end{eqnarray*}
    Similarly, we define
    \begin{eqnarray}\label{lemh2-06}
        \tilde{p}_{12}(x,y) = \frac{1}{\tilde{Z}_n^{(1,2)}} \exp \Big(  \frac{\beta }{n} \inprod{x+y, \sigma^{(1,2)}}   \Big) ,
    \end{eqnarray}
    where
    $x,y\in \ss^2$ and
    \begin{eqnarray*}
        \tilde{Z}_n^{(1,2)} =  \int_{\ss^2 } \int_{\ss^2}  \exp \Big(  \frac{\beta }{n} \inprod{x+y, \sigma^{(1,2)}}   \Big)  dP_n(x) dP_n(y).
    \end{eqnarray*}
    For any $x,y\in \ss^2$, we have
    \begin{eqnarray}\label{lemh2-07}
        | p_{12}(x,y) - \tilde{p}_{12}(x,y) | \leq C n^{-1}.
    \end{eqnarray}
    Let $(\xi_1,\xi_2)$ be a random vector with conditional density function $\tilde{p}_{12}(x,y)$,  given $\sigma^{(1,2)}$. % with % \sim \tilde{p}_{12}(x,y)$ given $\sigma^{(1,2)}$.
    Then,  for the first term of \eqref{lemh2-04},  by  \eqref{lemh2-07}, we have
    \begin{eqnarray}\label{lemh2-08}
        \lefteqn{ \nonumber \E \big( \inprod{\sigma_1, \sigma^{(1,2)}}^2  - V_1 (\sigma^{(1,2)})  \big) \big( \inprod{\sigma_2, \sigma^{(1,2)}}^2  - V_2 (\sigma^{(1,2)}) \big)   }\\
        &=& \E \big( \inprod{\xi_1, \sigma^{(1,2)}}^2  - \tilde{V}_1 (\sigma^{(1,2)})  \big) \big( \inprod{\xi_2, \sigma^{(1,2)}}^2  - \tilde{V}_2 (\sigma^{(1,2)}) \big)  + L_2,
    \end{eqnarray}
    where $|L_2| \leq C n^{3}$ and for $i=1,2$,
    \begin{eqnarray}\label{lemh2-v1}
        \tilde{V}_i (\sigma^{(1,2)}) &=&  \conep{ \inprod{\xi_i, \sigma^{(1,2)}}^2 }{\sigma^{(1,2)}} \nonumber\\  &=&  | \sigma^{(1,2)}|^2 \Big( 1 - \frac{2\psi(b_{12})}{b_{12}} \Big) ,\\
        b_{12} &=&  \beta |\sigma^{(1,2)}| /n. \label{lemh2-v2}
    \end{eqnarray}
    Observe that given $\sigma^{(1,2)}$, $\xi_1$ and $\xi_2$ are conditionally independent;  then,  the first term of \eqref{lemh2-08} is 0, and thus,
    \begin{eqnarray}\label{lemh2-v3}
        \Big| \E \big( \inprod{\sigma_1, \sigma^{(1,2)}}^2  - V_1 (\sigma^{(1,2)})  \big) \big( \inprod{\sigma_2, \sigma^{(1,2)}}^2  - V_2 (\sigma^{(1,2)}) \big) \Big|  \leq C n^3.
    \end{eqnarray}
    It suffices to bound the second term of \eqref{lemh2-04}. Again, by  \eqref{lemh2-07}, we have
    \begin{eqnarray}
        \label{lemh2-09}
        \lefteqn{   \E \big( V_1 (\sigma^{(1,2)}) - a \big) \big( V_2 (\sigma^{(1,2)}) - a \big) \nonumber   }\\
        &=&  \E \big( \tilde{V}_1 (\sigma^{(1,2)}) - a \big) \big( \tilde{V}_2 (\sigma^{(1,2)}) - a \big) + L_3  ,
    \end{eqnarray}
    where $|L_3 |\leq C n^{3}$.
    Recalling that $\beta \psi(\kappa) = \kappa $ and the definition of $a$, we obtain
    \begin{eqnarray}
        \label{lemh2-10}
        \lefteqn{ \Big|  \tilde{V}_1(\sigma^{(1,2)}) - a \Big| \nonumber }\\
        &\leq & | \sigma^{(1,2)}|^2 \Big| \frac{\psi(b_{12})}{b_{12}} - \frac{\psi(\kappa)}{\kappa} \Big| + \Big( 1- \frac{2}{\beta}  \Big)  \Big| | \sigma^{(1,2)}|^2 - \frac{(n-1)^2 \kappa^2 }{\beta^2} \Big| \nonumber\\
        & \leq & C n^2 |b_{12} - \kappa| + C n.
    \end{eqnarray}
    By   \eqref{lemh2-10} and similar to \eqref{hl-3}, we have
    \begin{align}
        \label{lemh2-11}
        \ML   \Big| \E \big( \tilde{V}_1 (\sigma^{(1,2)}) - a \big) \big( \tilde{V}_2 (\sigma^{(1,2)}) - a \big) \Big| \nonumber\\ & \leq  C n^{4} \E |b_{12} - \kappa |^2 + C n^3 \nonumber\\ & \leq   C n^{3}.
    \end{align}
    It follows from \eqref{lemh2-09} and \eqref{lemh2-11} that
    \begin{eqnarray}\label{lemh2-12}
        \Big|   \E \big( V_1 (\sigma^{(1,2)}) - a \big) \big( V_2 (\sigma^{(1,2)}) - a \big)  \Big| \leq C n^{3}.
    \end{eqnarray}
    The inequalities \eqref{lemh2-03}, \eqref{lemh2-04}, \eqref{lemh2-v3} and \eqref{lemh2-12} yield $|H_2| \leq C n^{5}$, and this completes the proof together with \eqref{lemh2-01} and \eqref{lemh2-02}.
\end{proof}
Next, we give the proof of \eqref{lemh3-1}.

\begin{proof}[Proof of \eqref{lemh3-1}]

\label{page-lemh3-1}
    Given the  symmetry, we have
    \begin{eqnarray}
        \label{lemh3-01}
        \E \Big( \sum_{i=1}^n M_i \Big) ^2 = n \E (M_1^2) + n(n-1) \E (M_1 M_2).
    \end{eqnarray}
    As $|\sigma^{(1)}| \leq n$, we have $\ep{M_1^2} \leq C n^{4}$. For $\ep{M_1M_2}$,
    we define
    \begin{eqnarray*}
        m_i &=& \inprod{\sigma_i, \sigma^{(i)}} |\inprod{\sigma_i,\sigma^{(i)}} |,\nonumber\\ % - \conep{\inprod{\sigma_i, \sigma^{(i)}} |\inprod{\sigma_i,\sigma^{(i)}} |}{\sigma^{(i)}} , \nonumber\\
        m_i^{(1,2)} &=&  \inprod{\sigma_i, \sigma^{(1,2)}} |\inprod{\sigma_i,\sigma^{(1,2)}} |, % - \conep{\inprod{\sigma_i, \sigma^{(1,2)}} |\inprod{\sigma_i,\sigma^{(1,2)}} |}{\sigma^{(1,2)}}
    \end{eqnarray*}
    where $i=1,2$.
    Then,  we have $|m_i - m_i^{(1,2)}| \leq C n.$ Thus,
    \begin{eqnarray}
        \label{lemh3-02}
        \ep{M_1M_2} = \ep{M_1^{(1,2)} M_{2}^{(1,2)}} + L_4,
    \end{eqnarray}
    where $|L_4| \leq C n^{3}$ and
    \begin{eqnarray*}
        M_i ^{(1,2)} = m_i^{(1,2)} - \conep{m_i^{(1,2)}}{\sigma^{(1,2)}}.
    \end{eqnarray*}
    Let $(\xi_1,\xi_2)$ be as defined in \eqref{lemh2-08}. By  \eqref{lemh2-05}--\eqref{lemh2-07}, we have
    \begin{eqnarray}
        \label{lemh3-03}
        \Big| \ep{M_1^{(1,2)}  M_2^{(1,2)}} - \ep{\tilde{M}_1^{(1,2)} \tilde{M}_2^{(1,2)}} \Big|\leq C n^3,
    \end{eqnarray}
    where for $i=1,2$,
    \begin{eqnarray*}
        \tilde{M}_i^{(1,2)} &=&  \tilde{m}_i^{(1,2)} - \conep{\tilde{m}_i^{(1,2)}}{\sigma^{(1,2)}},\nonumber\\
        \tilde{m}_i^{(1,2)} &=&  \inprod{\xi_i, \sigma^{(1,2)}} |\inprod{\xi_i,\sigma^{(1,2)}} |.
    \end{eqnarray*}
    As  $\xi_1$ and $\xi_2$ are conditionally independent given  $\sigma^{(1,2)}$, we have  $$ \E(\tilde{M}_1^{(1,2)} \tilde{M}_2^{(1,2)}) =0,  $$ and by  \eqref{lemh3-02} and \eqref{lemh3-03} we have
    $|\ep{M_1M_2}| \leq C n^3.$ Together with \eqref{lemh3-01}, we complete the proof of \eqref{lemh3-1}.
\end{proof}
\subsection{Proof of \cref{thm:me}}
As  the vertices are colored independently and uniformly, we can
construct the exchangeable pair as follows. Let $\xi_i',\cdots,\xi_n'$ be
independent copies of $\xi_1,\cdots,\xi_n$, and $I$
be a random index
independent
of all others and uniformly distributed over $\left\{ 1,\cdots,n \right\}$.
Recall that
$$
W : = W_n = \frac{1}{2} \sum_{i = 1}^n \sum_{j \in A_i} \frac{  \IN{\xi_i = \xi_j } - \frac{1}{c_n}}{ \sqrt{ \frac{m_n}{c_n} (1 - \frac{1}{c_n})}}.
$$
We
replace $\xi_I$ with $\xi_I'$ in $W$ to obtain a new random variable
$W'$; then,  $(W,W')$ is an exchangeable pair.
Let  $\X$ be the sigma field generated by
$\left\{ \xi_1,\cdots,\xi_n \right\}$ and $\sigma^2 = {m_n\over c_n}(1-{1\over
c_n})$.
We have
\begin{eqnarray*}
  \conep{W-W'}{\X} &=& {1\over n} \sum_{i=1}^n \sum_{j\in A_i} {\IN{\xi_i
  =\xi_j} - \conep{\IN{\xi_i' = \xi_j}}{\X} \over \sigma}\\
  &=& {1\over n} \sum_{i=1}^n \sum_{j\in A_i} {\IN{\xi_i
  =\xi_j} -1/c_n \over \sigma}\\
  &=& {2\over n} W.
\end{eqnarray*}
Hence, \eqref{con1} holds with $\lambda ={2\over n}$ and $R_n =0$.
By  \cref{thm-1}, it suffices to prove
\begin{eqnarray}
   \lefteqn{  \E\left|  1 -  {1\over 2\lambda} \conep{(W-W')^2 }{W}  \right|}\nonumber\\  & \leq &
  C( \sqrt{1/c_n} + \sqrt{d^*_n/m_n} + \sqrt{c_n/m_n}) \label{me2}
  \end{eqnarray}
  and
  \begin{eqnarray}
  \lefteqn{ {1\over \lambda}\E\left|  \conep{(W-W')|W-W'|}{W}  \right|}\nonumber\\ & \leq &
 C( \sqrt{d^*_n/m_n} + \sqrt{c_n/m_n}),
  \label{me3}
\end{eqnarray}
where $C$ is an absolute constant and $d_n^* = \max \{ d_i, 1 \le i \leq n \}$.
\begin{proof}
  [Proof of \eqref{me2}]
  Observe that
  \begin{eqnarray}\label{me2-1}
    \lefteqn{\conep{(W-W')^2}{\X}}\nonumber\\
    &=& {1\over n\sigma^2 } \sum_{i=1}^n \conepb{ \Big( \sum_{j\in A_i}
    \IN{\xi_i=\xi_j} - \IN{\xi_i' = \xi_j} \Big)^2  }{\X}\nonumber \\
    &=& {1\over n\sigma^2 } \sum_{i=1}^n \Big( \Big( \sum_{j\in A_i}
    (\IN{\xi_i=\xi_j}-1/c_n) \Big)^2 + \conepb{ \Big( \sum_{j\in A_i}
\IN{\xi_i'=\xi_j} -1/c_n \Big)^2 }{\X} \Big)\nonumber\\
&=& \frac{1}{n\sigma^2} \sum_{i=1}^n \Big( \Big( \sum_{j \in A_i } h(\xi_i,\xi_j) \Big) ^2 + \conepb{ \Big(\sum_{j \in A_i} h(\xi_i',\xi_j) \Big) ^2  }{\X} \Big) ,\nonumber\\
  \end{eqnarray}
  where
  \begin{eqnarray*}
      h(x,y) = \IN{x=y} - 1/c_n.
  \end{eqnarray*}
  By
  the law of total variance, we need only to bound the variance of the
  first
  term. Note that
  \begin{eqnarray}\label{me2-2}
    \lefteqn{ \Var\Big( \sum_{i=1}^n \Big( \sum_{j\in A_i}
    h(\xi_i,\xi_j) \Big)^2 \Big)  }\nonumber\\
    &=& \sum_{i=1}^n \Var\Big( \sum_{j\in A_i} h(\xi_i,\xi_j)
    \Big)^2 \nonumber \\
    && + \sum_{i\neq i'} \Cov\Big( \Big( \sum_{j\in A_i} h(\xi_i,\xi_j)
        \Big)^2, \Big( \sum_{l\in A_{i'}} h(\xi_{i'}, \xi_l).
    \Big)^2   \Big)
  \end{eqnarray}
 As
 $$\Big(\sum_{j\in A_i} h(\xi_i,\xi_j) \Big)^2 = \sum_{j\in A_i}
 h^2(\xi_i,\xi_j) +\sum_{j\neq l \in A_i}
 h(\xi_i,\xi_j) h(\xi_i,\xi_l),
$$
we have
\begin{eqnarray*}
    \Var \Big( \sum_{j \in A_i} h(\xi_i,\xi_j) \Big)^2 \leq 2 \Var \Big( \sum_{j\in A_i}
    h^2(\xi_i,\xi_j) \Big)  + 2 \Var\Big( \sum_{j\neq l \in A_i}
    h(\xi_i,\xi_j) h(\xi_i,\xi_l) \Big) .
\end{eqnarray*}

Note that
  \begin{eqnarray*}
      \lefteqn{\Var\Big( \sum_{j\in A_i}h^2(\xi_i,\xi_j)   \Big)}\\
      &=& \epb{\Var\Big(\sum_{j\in A_i} h^2(\xi_i,\xi_j) \Big| \xi_i
    \Big)}
   + \Var\Big(\conepb{\sum_{j\in A_i}h^2(\xi_i,\xi_j)}{\xi_i}
    \Big)\\
    &=& d_i \left( {1\over c_n}\left( 1-{1\over c_n} \right)\left( 1-{2\over c_n
    }+{2\over c_n^2} \right) \right)\\
    & \leq & d_i /c_n,
  \end{eqnarray*}
  where for every $i \neq j$,
  \begin{align}
      \Var( h^2(\xi_i,\xi_j) \mid \xi_i) & =
      (1/c_n)(1-1/c_n)(1-2/c_n+2/c_n^2),  \label{var-1}\\ \intertext{and}
  \conep{ h^2(\xi_i,\xi_j)}{\xi_i} &
  =(1/c_n)(1-1/c_n).  \label{var-2}
  \end{align}
  Also, for $j\neq l\neq i$,
  $\ep{h(\xi_i,\xi_j)h(\xi_i,\xi_l)}=0$. Thus,  we have
  \begin{eqnarray*}
      \lefteqn{ \Var\Big(\sum_{j\neq l\in A_i} h(\xi_i,\xi_j) h(\xi_i,\xi_l)
       \Big)  }\\
       &=& \epb{\sum_{j\neq l\in A_i} h(\xi_i,\xi_j)h(\xi_i,\xi_l)  }^2\\
    &=& 2d_i (d_i-1) \left( {1\over c_n}\left( 1-{1\over c_n} \right)
    \right)^2\\
    & \leq & 2d_i^2 / c_n^2.
  \end{eqnarray*}
  Therefore,
  \begin{eqnarray}\label{me2-5}
      \Var\Big(\sum_{j\in A_i} h(\xi_i,\xi_j)   \Big)^2 \leq 4 d_i/c_n + 4 d_i^2/c_n^2.
  \end{eqnarray}
  This gives the bound of the first term of \eqref{me2-2}. To bound the second term of \eqref{me2-2},
  we let $\delta_{ii'} = \IN{(v_i,v_{i'})\in E}$ for $i\neq i'$, which indicates the connection between vertex $i$ and $i'$.  We have
  \begin{eqnarray}\label{me2-3}
      \lefteqn{\Cov\Big( \Big(\sum_{j\in A_i} h(\xi_i,\xi_j)
    \Big)^2,
    \Big(\sum_{l\in A_{i'}} h(\xi_{i'},\xi_l)  \Big)^2  \Big)}\nonumber\\
    &=&\nonumber \Cov \Big( \sum_{j \in A_i} h^2(\xi_i,\xi_j) + \sum_{j\neq j'\in A_i } h(\xi_i,\xi_{j})h(\xi_i, \xi_{j'}) , \\ && \nonumber \qquad \sum_{l \in A_{i'}} h^2(\xi_{i'}, \xi_l)+ \sum_{l\neq l' \in A_{i'}} h(\xi_{i'}, \xi_l) h(\xi_{i'}, \xi_{l'}) \Big) \\
    &=&\nonumber \sum_{j \in A_i} \sum_{l \in A_{i'}} \Cov(h^2(\xi_i,\xi_j), h^2(\xi_{i'},\xi_l)) \\ \nonumber &&   + \sum_{j \neq j'\in A_i} \sum_{l \in A_{i'}} \Cov(h(\xi_i,\xi_j) h(\xi_i,\xi_{j'}), h^2(\xi_{i'},\xi_l) )\\
    &&\nonumber + \sum_{j \in A_i} \sum_{l\neq l' \in A_{i'}} \Cov( h^2(\xi_i,\xi_j), h(\xi_{i'},\xi_l)h(\xi_{i'},\xi_{l'}) ) \\ && \nonumber + \sum_{j\neq j'\in A_i}\sum_{l\neq l' \in A_{i'}} \Cov( h(\xi_i,\xi_j)h(\xi_i,\xi_{j'}), h(\xi_{i'},\xi_l)h(\xi_{i'},\xi_{l'}) )\\
    &:=& H_1 + H_2 + H_3 + H_4.
  \end{eqnarray}
  %where we expand the second line and calculate each of the covariance to
  %get
  %the last line.
  Next, we compute the preceding  covariances.
  For $H_1$, we have
  \begin{eqnarray}%\label{me2-h1}
      H_{1} &=& \delta_{ii'} \Var(h^2(\xi_i,\xi_{i'})) + \delta_{ii'}\sum_{j  \in A_{i} \setminus \{i'\}} \Cov( h^2(\xi_i, \xi_{i'}), h^2(\xi_{i}, \xi_j ) )\nonumber\\
      &&  +  \delta_{ii'} \sum_{l  \in A_{i'} \setminus \{i\}} \Cov(h^2(\xi_i, \xi_{i'}), h^2(\xi_{i'},\xi_l))\nonumber\\ &&  + \sum_{j  \in A_{i} \setminus \{i'\}} \sum_{l  \in A_{i'} \setminus \{i\}} \Cov(h^2(\xi_i,\xi_j), h^2(\xi_{i'}, \xi_l) ).\nonumber
     % &=& \sum_{j \in A_i \cap A_{i'}} \Cov(h^2(\xi_i,\xi_j), h^2(\xi_{i'},\xi_j))\nonumber\\
      %&=&\sum_{j \in A_i \cap A_{i'}} \Cov \big(  \conep{h^2(\xi_i,\xi_j)}{\xi_j}, \conep{h^2(\xi_{i'},\xi_j)}{\xi_j}  \big) \nonumber\\
      %&=& 0.
  \end{eqnarray}
  For the first term, by  \eqref{var-1} and \eqref{var-2}, %we know
  we have
  \begin{eqnarray*}
      \Var(h^2(\xi_i,\xi_{i'})) \leq 1/c_n.
  \end{eqnarray*}
  For $j \in A_i \setminus \left\{ i' \right\}$, by  \eqref{var-2}, we have
  \begin{eqnarray*}
      \Cov(h^2(\xi_i,\xi_{i'}), h^2(\xi_i,\xi_j)) &=& \Cov ( \conep{h^2(\xi_i,\xi_{i'})}{\xi_i}, \conep{h^2(\xi_i,\xi_j)}{\xi_i} )\nonumber\\
      &=& 0.
  \end{eqnarray*}
  Similarly, for $l \in A_{i'}\setminus \left\{ i \right\}$, we have
  \begin{eqnarray*}
      \Cov(h^2(\xi_{i'},\xi_i), h^2(\xi_{i'}, \xi_l)) = 0.
  \end{eqnarray*}
  For the last term, if $j \neq l \nin \left\{ i,i' \right\}$, then $h(\xi_i,\xi_j)$ and $h(\xi_{i'},\xi_l)$ are independent. %and thus by \eqref{var-2},
  If $j = l \nin \left\{ i,i' \right\}$, by  \eqref{var-2}, we have
  \begin{eqnarray*}
      \Cov(h^2(\xi_i, \xi_j), h^2(\xi_{i'}, \xi_l)) =0.
  \end{eqnarray*}
  Therefore,
  \begin{eqnarray}
      | H_1 | \leq \delta_{ii'} / c_n.
      \label{me2-h1}
  \end{eqnarray}

  For $H_2$,
  we have
  \begin{eqnarray}\label{meh2-a}
      H_2 &=& \delta_{ii'} \sum_{j \neq j' \in A_i} \Cov(h(\xi_i,\xi_{j})h(\xi_i, \xi_{j'}), h^2(\xi_i,\xi_{i'})) \nonumber\\
      && + \sum_{ j \neq j' \in A_i } \sum_{ l \in A_{i'} \setminus \left\{ i \right\}} \Cov (h(\xi_i,\xi_{j})h(\xi_i, \xi_{j'}), h^2(\xi_i,\xi_{l}))\nonumber\\
      &=& H_{21} + H_{22}.
  \end{eqnarray}
  For $H_{21}$, if $j \neq i'$ or $j' \neq  i'$, then
  \begin{eqnarray*}
      \lefteqn{  \Cov(h(\xi_i,\xi_{j})h(\xi_i, \xi_{j'}), h^2(\xi_i,\xi_{i'}))   }\\
      &=& \ep{  h(\xi_i,\xi_{j})h(\xi_i, \xi_{j'}) h^2(\xi_i,\xi_{i'})}\\
      &=& \ep{ \conep{h(\xi_i,\xi_{j})h(\xi_i, \xi_{j'}) h^2(\xi_i,\xi_{i'})}{\xi_i,\xi_{i'}} } \\
      &=& 0.
  \end{eqnarray*}
  If $j = i'$ or $j'= i$, similarly,
  \begin{eqnarray*}
      \Cov( h(\xi_i,\xi_{j})h(\xi_i, \xi_{j'}), h^2(\xi_i,\xi_{i'}))=0.
  \end{eqnarray*}
  Therefore,
  \begin{eqnarray}\label{meh2-b}
      H_{21 } =0.
  \end{eqnarray}
  For $H_{22}$, %we have
the covariance is not zero
only if $\left\{ j,j' \right\} = \left\{ i',l \right\}$.  Therefore,
\begin{eqnarray}\label{meh2-c}
      H_{22} &=&  \sum_{l \in A_i \cap A_{i'}} \Cov (h(\xi_i, \xi_{i'})h(\xi_i,\xi_l), h^2(\xi_{i'},\xi_l) )\nonumber\\
      &=&  \sum_{l \in A_i \cap A_{i'}}  \ep{  \conep{ h(\xi_i, \xi_{i'})h(\xi_i,\xi_l), h^2(\xi_{i'},\xi_l)  }{\xi_{i'},\xi_l} }\nonumber\\
      &=& \frac{1}{c_n}  \sum_{l \in A_i \cap A_{i'}} \ep{h^3 (\xi_{i'},\xi_l) } \nonumber\\
      & \leq &  C (d_i \wedge d_{i'})/c_n^2.
  \end{eqnarray}
  Similarly, $H_{22} \geq - C (d_i \wedge d_{i'})/c_n^2 $.
  By \eqref{meh2-a}--\eqref{meh2-c},
  \begin{eqnarray}
      \label{me2-h2}
      |H_2| \leq  C (d_i \wedge d_{i'}) /c_n^2.
  \end{eqnarray}
  Similarly,
  \begin{eqnarray}
      \label{me2-h3}
      |H_3| \leq  C (d_i \wedge d_{i'}) /c_n^2.
      %H_3= 0.
  \end{eqnarray}
  For $H_4$, we have
  \begin{eqnarray*}
      H_4 &=& 2 \delta_{ii'} \sum_{j \in A_i \setminus \left\{ i' \right\} } \sum_{ l \neq l' \in A_{i'}  } \Cov ( h(\xi_i,\xi_{i'})h(\xi_i,\xi_j ) , h(\xi_{i'} ,\xi_l) h(\xi_{i'},\xi_{l'})  )\nonumber\\
      && + \sum_{ j\neq j' \in A_i \setminus \left\{ i' \right\} } \sum_{l \neq l' \in A_{i'} } \Cov( h(\xi_i, \xi_{j})h(\xi_i,\xi_{j'}), h(\xi_{i'},\xi_l) h(\xi_{i'},\xi_{l'}) )\nonumber\\
      &:=& H_{41} + H_{42}.
  \end{eqnarray*}
  For $H_{41}$,  the covariance is not zero only if $\left\{ l,l' \right\} = \left\{ i,j \right\}$. %, the covariance is not 0. Thus
  Thus,
  \begin{eqnarray*}
      |H_{41}| &=& 4 \delta_{ii'} \Bigl| \sum_{j \in A_i \cap A_{i'}} \Cov( h(\xi_i,\xi_{i'})h(\xi_i,\xi_j), h(\xi_{i'},\xi_i)h(\xi_{i'},\xi_j) ) \Bigr| \nonumber\\
      & \leq &  C \delta_{ii'} (d_i \wedge d_{i'}) / c_n^2.
  \end{eqnarray*}
  For $H_{42}$, the covariance is not zero only if $\left\{ j,j' \right\} = \left\{ l.l' \right\} $.
  \begin{eqnarray*}
      H_{42} &=&  2 \sum_{j \neq j' \in A_i \cap A_{i'}} \Cov(   h(\xi_i, \xi_{j})h(\xi_i,\xi_{j'}), h(\xi_{i'},\xi_j) h(\xi_{i'},\xi_{j'}) )\nonumber\\
      &=&  2 \sum_{j \neq j' \in A_i \cap A_{i'}}   \Cov\left(  \conep{ h(\xi_i, \xi_{j})h(\xi_i,\xi_{j'})}{\xi_j,\xi_{j'}} , \conep{ h(\xi_{i'},\xi_j) h(\xi_{i'},\xi_{j'})}{\xi_j,\xi_{j'}}  \right)\nonumber\\
      &=&  \frac{2}{c_n^2} \sum_{j \neq j' \in A_i \cap A_{i'}}  \Var( h(\xi_{j},\xi_{j'})  )\nonumber\\
      & \leq &  C (d_i \wedge d_{i'})^2 / c_n^3.
  \end{eqnarray*}

  Therefore,
  \begin{eqnarray}\label{me2-h4}
      | H_4|  \leq C \delta_{ii'} (d_i\wedge d_{i'}) /c_n^2 + C (d_i\wedge d_{i'})^2 / c_n^3.
  \end{eqnarray}
  Combining \eqref{me2-3}, \eqref{me2-h1}, \eqref{me2-h2}, \eqref{me2-h3} and \eqref{me2-h4} we have
  \begin{eqnarray}\label{me2-4}
     \lefteqn{\nonumber  \Cov\Big( \Big(\sum_{j\in A_i} h(\xi_i,\xi_j)
    \Big)^2,
\Big(\sum_{l\in A_{i'}} h(\xi_{i'},\xi_l)  \Big)^2  \Big)}\\ &  \leq &  C \Big( \delta_{ij} / c_n + (d_i \wedge d_{i'})/c_n^2 + (d_i\wedge d_{i'})^2 / c_n^3\Big) .
  \end{eqnarray}
  By \eqref{me2-2}, \eqref{me2-5} and \eqref{me2-4}, we have
  \begin{eqnarray*}
      \lefteqn{  \Var \Big( \sum_{i=1}^n \Big( \sum_{j \in A_i} h(\xi_i,\xi_j) \Big) ^2 \Big)   }\\
      %& \leq & C \Big( \sum_{i=1}^n (d_i / c_n + d_i^2 / c_n^2 ) + \sum_{i \neq i'}\delta_{ii'} (d_i \wedge d_{i'})/c_n^2 \Big)\\
      & \leq &  C( d_n^* m_n/c_n^2 + m_n / c_n + m_n^2 / c_n^3).
  \end{eqnarray*}

  The law of total variance yields
  %that
  \begin{eqnarray*}
    \Var\Bigg( \sum_{i=1}^n \conepb{ \Big(\sum_{j\in A_i}
                h(\xi_{i}', \xi_j)
    \Big)^2}{\X}    \Bigg)\leq
    C\Big(  {d^*_n m_n\over c_n^2} + {m_n\over c_n} + \frac{m_n^2 }{ c_n^3} \Big),
  \end{eqnarray*}
  and thus,
  \begin{eqnarray*}
    \lefteqn{ \Var\Big({1\over 2\lambda}\conep{(W-W')^2}{\X}\Big)  }\\
    & \leq & {C\over \sigma^4 } \Big(  {d^*_n m_n\over c_n^2} + {m_n\over c_n} + \frac{m_n^2}{c_n^3}
    \Big)\\
    & \leq & C (d^*_n/m_n + c_n/m_n + 1/c_n).
  \end{eqnarray*}
  This completes the proof of \eqref{me2}.
\end{proof}
\begin{proof}
  [Proof of \eqref{me3}]
  This proof is slightly different from that of \eqref{me2}. Observe
  that
  \begin{eqnarray*}
    \lefteqn{  \conep{(W-W')|W-W'|}{\X}  }\\
    &=& {n\sigma^2 } \sum_{i=1}^n \conepb{ \Big(\sum_{j\in A_i}
    \IN{\xi_i=\xi_j} - \IN{\xi_i'=\xi_j}    \Big)  \Big|\sum_{j\in A_i}
    \IN{\xi_i=\xi_j} - \IN{\xi_i'=\xi_j}    \Big| }{\X}.
  \end{eqnarray*}
  The variance of the preceding summation can be expanded to
  \begin{eqnarray*}
      \Var\Big(\sum_{i=1}^n  M_i  \Big)
      = \sum_{i=1}^n \Var(M_i) + \sum_{i\neq i'} \Cov(M_i, M_{i'}),
  \end{eqnarray*}
  where %for each $1\leq i\leq n$,
  \begin{eqnarray*}
    M_i =   \conepb{ \Big(\sum_{j\in A_i}
    \IN{\xi_i=\xi_j} - \IN{\xi_i'=\xi_j}    \Big)  \Big|\sum_{j\in A_i}
    \IN{\xi_i=\xi_j} - \IN{\xi_i'=\xi_j}    \Big| }{\X}.
  \end{eqnarray*}
  Noting that $\ep{M_i}=0,$ we have
  \begin{eqnarray*}
    \lefteqn{  \Var(M_i) = \ep{M_i^2}  }\\
    &\leq & \epb{ \Big(\sum_{j\in A_i} \IN{\xi_i=\xi_j} -
    \IN{\xi_i'=\xi_j}\Big)^4
    }\\
    & \leq & C d_i\left(
    {1\over c_n}\left( 1-{1\over c_n} \right)
    \right) \left( 2d_i\left( {1\over c_n}-{1\over c_n^2} \right)+1 \right).
  \end{eqnarray*}

  To calculate the covariance term, for each $i\neq j $, let $\eta_{ij} = \IN{\xi_i=\xi_j} - \IN{\xi_i' = \xi_j}$, % and
  \begin{equation*}
  \begin{array}{ccc}
      T_i = \sum_{j\in A_i } \eta_{ij}, & \text{and }&
    T_i^{(i')} = \sum_{j \in A_i\setminus
    \{i'\}}\eta_{ij}.
  \end{array}
\end{equation*}
Then, $M_i = \conep{ T_i|T_i|  }{\X}$.

Observe that for $i\neq i'$ and given that $\X$, $T_i|T_i|$ is a function of $\xi_i'$ and $T_{i'}|T_{i'}|$ is a function of $\xi_{i'}^{'}$; thus,
$\Cov(T_i|T_i|,T_{i'}|T_{i'}||\X) =0$. By the total covariance formula, we
have $\Cov(M_i,M_{i'})=\Cov(T_i|T_i|,
T_{i'}|T_{i'}|)$.
As $\xi_i$ and $\xi_i'$ are independent and identically distributed, % and thus,
$T_i|T_i|$ and $-T_i|T_i|$ are also identically distributed. Therefore, $\ep{T_i|T_i|}=0$, and for some constant $C$, we have
\begin{eqnarray*}
    \lefteqn{ \Cov(M_i, M_{i'})}\\
    &=&   \ep{T_i|T_i| T_{i'}|T_{i'}|}  \\
  &=& \ep{T_{i}^{({i'})}|T_i^{({i'})}| T_{i'}^{(i)}|T_{i'}^{(i)}|} + \ep{
  T_{i}^{({i'})}|T_i^{({i'})}| (T_{i'}|T_{i'}| -\delta_{ii'}T_{i'}^{(i)}|T_{i'}^{(i)}|)  }\\
  && + \ep{(T_i|T_i|-\delta_{ii'}T_{i}^{({i'})}|T_i^{({i'})}|) T_{i'}^{(i)}|T_{i'}^{(i)}|} \\ && + \ep{
  (T_i|T_i|-\delta_{ii'}T_{i}^{({i'})}|T_i^{({i'})}|) (T_{i'}|T_{i'}| -\delta_{ii'}T_{i'}^{(i)}|T_{i'}^{(i)}|)  }.
  %&=&  \ep{
%  (T_i|T_i|-T_{i}^{({i'})}|T_i^{({i'})}|) (T_{i'}|T_{i'}| -T_{i'}^{(i)}|T_{i'}^{(i)}|)  }\\
%  & \leq & C\delta_{ii'}(d_i d_{i'})^{1/2}\left( {1\over c^2}\left( 1-{1\over c}
%  \right)\left(
%    1-{2\over c} \right)
%    \right)+ C\delta_{i{i'}}\times {1\over c}\left( 1-{1\over c} \right)\left( 1+{2\over c}
%    -{2\over c^2} \right).
\end{eqnarray*}
Define $\F_{i} = \sigma \{ \xi_j, j \neq i\}$. Given  $\F_{i}$, $T_i|T_i|$ and $T_{i'}^{(i)}| T_{i'}^{(i)} |$ are conditionally independent, %/By
%then
\begin{eqnarray*}
  \epb{ T_i|T_i| T_{i'}^{(i)}| T_{i'}^{(i)} | } = \epb{ T_{i'}^{(i)}| T_{i'}^{(i)} | \conepb{ T_i|T_i|   }{\F_i} } = 0.
\end{eqnarray*}
Similarly,
\begin{eqnarray*}
    \epb{ T_{i'}|T_{i'}| T_{i}^{(i')}| T_{i}^{(i')} | }=0,
\end{eqnarray*}
and
\begin{eqnarray*}
  \ep{T_{i}^{({i'})}|T_i^{({i'})}| T_{i'}^{(i)}|T_{i'}^{(i)}|}=0.
\end{eqnarray*}
Thus,
\begin{eqnarray}
    \lefteqn{  \ep{T_i|T_i| T_{i'}|T_{i'}|}  }\nonumber\\
  &=&  \ep{
  (T_i|T_i|-\delta_{ii'}T_{i}^{({i'})}|T_i^{({i'})}|) (T_{i'}|T_{i'}| -\delta_{ii'}T_{i'}^{(i)}|T_{i'}^{(i)}|)  } .  \label{ine-me1}
  %\\
%  & \leq & C\delta_{ii'}(d_i d_{i'})^{1/2}/c^2 + C\delta_{i{i'}}/c
\end{eqnarray}
%where we used Cauchy inequality in the last line.
Without loss of generality, we assume that $\delta_{ii'} =1$. Note that
\begin{eqnarray*}
  \lefteqn{\big| T_i|T_i| - T_i^{(i')}|T_i^{(i')} |  \big|}\\
  & = & \big|(T_i - T_i^{(i')}) |T_i| + T_i^{(i')} (|T_i| - |T_i^{(i)}|)\big|\\
  & \leq & 2 |\eta_{ii'}T_i^{(i')}| + |\eta_{ii'}^2|,
\end{eqnarray*}
and
thus,
\begin{eqnarray*}
   \lefteqn{\E\big( T_i|T_i| - T_i^{(i')}|T_i^{(i')} |  \big)^2 }\\
   & \leq & C \epb{ \eta_{ii'}^2 \big( T_i^{(i)} \big)^2  } + C \ep{\eta_{ii'}^4}\\
   & = & C \Big( \sum_{j \in A_i\setminus\{i'\}} \ep{  \eta_{ii'}^2\eta_{ij}^2 } +  \sum_{j \neq l \in A_i\setminus\{i'\}} \ep{  \eta_{ii'}^2\eta_{ij}\eta_{il} }  + \ep{\eta_{ii'}^4}\Big)\\
   & \leq  & C d_i /c_n^2 + C / c_n.
\end{eqnarray*}
Similarly,
\begin{eqnarray*}
  \E(T_{i'}|T_{i'}| -\delta_{ii'}T_{i'}^{(i)}|T_{i'}^{(i)}|)^2 \leq C d_{i'}/c_n^2 + C/c_n.
\end{eqnarray*}
By  \eqref{ine-me1} and the  Cauchy inequality, we finally have
\begin{eqnarray*}
    |  \ep{T_i|T_i| T_{i'}|T_{i'}|}| \leq   C \sqrt{d_i  d_{i'}} /c_n^2 + C / c_n.
\end{eqnarray*}

Similar to the proof of \eqref{me2}, we obtain the bound \eqref{me3}.
\end{proof}

{\bf Acknowledgements.} We thank one referee, an Associate Editor and the Editor for their helpful suggestions which led to a much improved presentation of the paper.

%\bibliographystyle{imsart-nameyear}
%\setcitestyle{numbers}
%\bibliography{exchbib}
%%

\end{document}